\documentclass[a4paper,fleqn]{cas-dc}
\usepackage[utf8]{inputenc}
\usepackage[T1]{fontenc}
\usepackage{amsmath,amssymb}
\usepackage{amsthm}
\usepackage{mathtools}   
\usepackage{enumitem}
\usepackage{booktabs}

\newcommand{\cE}{\mathcal{E}}
\newcommand{\cN}{\mathcal{N}}
\newcommand{\Si}{\Sigma}

\newcommand{\RR}{\mathbb{R}}
\newcommand{\CC}{\mathbb{C}}
\newcommand{\NN}{\mathbb{N}}

\newcommand{\la}{\lambda}

\newtheorem{definition}{Definition}[section]
\newtheorem{theorem}{Theorem}[section]
\newtheorem{lemma}{Lemma}[section]
\newtheorem{proposition}{Proposition}[section]
\newtheorem{corollary}[theorem]{Corollary}

\newtheorem{remark}{Remark}
\newtheorem{axioms}{Axioms}

\begin{document}
	

\title{Asymptotic monoids of words and logarithmic-prefix functionals
	for divergent arithmetic products}

\shorttitle{Asymptotic monoids and logarithmic-prefix functionals}
\shortauthors{A. \'Alvarez Cruz et al.}
\author[1]{A. \'Alvarez Cruz}\ead{amaury@ic.ufrj.br}
\author[2]{E. A. \'Alvarez Guti\'errez}\ead{esteban.gutierrez.1@cp2.edu.br}

\affiliation[1]{%
	organization={Instituto de Computa\c{c}\~ao},
	addressline={Universidade Federal do Rio de Janeiro (UFRJ)},
	city={Rio de Janeiro},
	country={Brazil}}

\affiliation[2]{%
	organization={Col\'egio Pedro II, Campus Centro},
	city={Rio de Janeiro},
	country={Brazil}}


\begin{abstract}
	We define and apply a generalized monoid of words, obtained as an
	asymptotic quotient of moderate nets of finite words, to the
	regularization of divergent products of arithmetic origin. Divergent
	sequences are encoded as moderate nets over a finite alphabet, while
	prefix-dependent observables are evaluated along admissible asymptotic
	scales. We characterize the corresponding evaluation and
	renormalization schemes axiomatically. Once the asymptotic scale and
	the class of admissible regular expansions are fixed, the
	constant-term prescription is uniquely determined. Thus, the
	regularized value is canonical relative to the chosen asymptotic scale
	and evaluation functional, without implying an absolute uniqueness
	principle.
	
	For the examples considered, the regularizing functionals arise as
	expectations of logarithmic-prefix functionals with respect to the
	empirical distributions of the associated sign oscillations. The
	alternating product of integers is thereby regularized to
	\(\sqrt{2/\pi}\), in agreement with zeta regularization, while the
	alternating product of factorials yields \((2/\pi)^{1/4}\) and
	satisfies the symmetry
	\(\mathfrak{F}_{\mathrm{ren}}=\sqrt{P_{\mathrm{ren}}}\). Products
	governed by dyadic block signs acquire finite values through
	logarithmic-scale averaging in the corresponding dyadic quotient, with
	invariance under the associated dyadic equivalence.
	
	For the Thue--Morse sequence, the zeta regularization is well defined
	through the meromorphic continuation of its Dirichlet series. A
	separate convergent functional associated with its logarithmic partial
	sums yields \(\exp(\mathcal F)\approx0.9732\). We also investigate
	numerically a substitution-generated sign sequence exhibiting exact
	square-root growth on dyadic blocks and a fractal log-periodic
	structure. Its Dirichlet series exhibits a candidate log-periodic pole
	structure and separates local scale constants from a candidate global
	regularized value. The choice of asymptotic scale is determined by the
	arithmetic structure of the underlying sequence: logarithmic scales for
	ordinary alternating products and dyadic scales for block-automatic
	sequences.
\end{abstract}



	\maketitle
	\section{Introduction}
	\label{sec:intro}
	
	The regularization of divergent infinite products has a long history in
	analysis and number theory.  Classical methods include Cesàro, Abel, and
	Riesz summation, as well as zeta regularization and related analytic
	continuation techniques \cite{Hardy1949,Titchmarsh1986}.  In many cases,
	formal products such as
	\[
	\prod_{n\ge1} n^{(-1)^{n+1}}
	\]
	are divergent classically but can be assigned finite values by a suitable
	renormalization scheme.
	
In this paper we use a generalized monoid of words, defined in
Section~\ref{sec:prelim} and summarized in Appendix~\ref{app:genmonoid},
as a systematic algebraic framework for evaluating prefix-dependent
functionals on words of controlled infinite length. The key point is
that this monoid is not a mere formal device: it gives a precise
meaning to the concatenation of infinite words arising from divergent
arithmetic products. In the classical theory of infinite words
\cite{PerrinPin2004,Lothaire2002}, the set of right-infinite words is
not naturally closed under concatenation: concatenating two
right-infinite words produces a two-sided sequence rather than a
right-infinite word. The generalized monoid introduced here restores
that closure by quotienting moderate nets of finite words under a
bidirectional prefix--suffix asymptotic equivalence. As a consequence,
infinite words may be concatenated while preserving a monoid
structure, and the operation is compatible with the evaluation of
logarithmic-prefix functionals.

The algebraic objects and proofs needed here are self-contained; the
companion paper \cite{AC1} contains additional structural developments
not used below.  The monoid is an asymptotic quotient of moderate nets
of finite words under a bidirectional prefix--suffix equivalence.  A
crucial feature of this quotient is that any functional depending only
on the logarithmic window
\(N_0(\varepsilon)=\lfloor\log_2(1/\varepsilon)\rfloor\) is well
defined on the quotient.  This provides a natural setting for encoding
arithmetic sequences as nets of words and for evaluating
logarithmic-prefix moulds. 

For sequences whose Dirichlet series admits a classical meromorphic
continuation, the generalized monoid yields the same regularized
values as a consistency check. Its actual purpose is to handle
sequences for which no such continuation is readily available,
especially those with a multiscale or log-periodic structure. The resulting algebraic framework is then
used to organize the regularization of divergent arithmetic products:
the examples treated in Section~\ref{sec:applications} show that the
same monoid structure accommodates products whose divergence is
logarithmic, dyadic, or arithmetic in different scales, and that the
choice of scale is dictated by the arithmetic structure of the sign
sequence.
	
	The aim of this paper is to develop the generalized monoid as a unified
algebraic framework for regularizing divergent arithmetic products and
sums, and for evaluating the logarithmic-prefix functionals attached to
them. The concrete results are as follows.
	
	\begin{itemize}
		\item The alternating product of integers is regularized to
		\(\sqrt{2/\pi}\), in exact agreement with the value obtained from the
		Dirichlet eta function at \(0\). This is a consistency check for the
		algebraic framework.
		\item The alternating product of factorials is regularized to
		\((2/\pi)^{1/4}\), and we show the symmetry
		\(\mathfrak{F}_{\mathrm{ren}}=\sqrt{P_{\mathrm{ren}}}\).
	\item Products with dyadic block signs, whose logarithmic partial
	sums have two distinct finite limit points, acquire a value through
	a logarithmic-scale average. This value coincides with the Cesàro
	limit of the logarithmic partial sums, but the dyadic-scale quotient
	\(\widetilde{\Si}^*_{\mathrm{dyad}}\), adapted to the block structure
	of the signs, gives that average an algebraic interpretation and
	dyadic invariance.
\item For the Thue--Morse sequence, the zeta regularization of the
divergent product \(\prod n^{t_n}\) is well defined via the
meromorphic continuation of the associated Dirichlet series. A
separate convergent auxiliary functional attached to the logarithmic
pial sums takes the value approximately \(0.9732\), and a precise
analytic relation between this functional and ordinary derivatives of
the associated Dirichlet series is proved.
	\item As a further illustration of the scale-detection principle,
we study a substitution-generated sign sequence satisfying
\(S(4^m)=2^m\). Its Dirichlet series has a candidate log-periodic pole hierarchy of powers
\(2^{-jm}\). This example, whose summatory function has a fractal
log-periodic structure, and separates
local scale constants from the global regularized value
\(P_{\mathrm{ren}}=\exp(-D'(0))\).
	\end{itemize}
Among these examples, the dyadic-scale quotient provides a natural algebraic realization of
the logarithmic-scale average, which is not obtained here from a
standard classical regularization; the dyadic-scale quotient is
essential there.

A guiding principle throughout is that each evaluation functional is
determined by the canonical oscillation measure of the underlying sign
sequence. This principle is formalized in
Section~\ref{sec:canonical_measures}.

Throughout the paper we distinguish carefully between functionals that
descend to the monoid quotient and those that require a chosen canonical
representative.

The paper is organized as follows.  Section~\ref{sec:prelim} defines the algebraic framework and states the
results needed in this paper; Appendix~\ref{app:genmonoid} provides the
corresponding proofs and a concise summary of the generalized monoid.
The companion paper \cite{AC1} may be consulted for further algebraic
developments and examples not used here.
Section~\ref{sec:applications} develops the regularization of alternating
products and the Thue--Morse product.  Section~\ref{sec:selection}
presents a heuristic classification of the divergence types encountered
in the examples; it is not intended as a general theory, but as a
practical guide for choosing evaluation functionals and renormalization
schemes.  Section~\ref{sec:chebyshev} treats the Chebyshev function as an
illustration, and Section~\ref{sec:dyadic} analyzes dyadic block products
using the dyadic-scale quotient \(\widetilde{\Si}^*_{\mathrm{dyad}}\).
Numerical illustrations are collected in Section~\ref{sec:numerical}.
	
\section{Preliminaries: the generalized monoid and logarithmic-prefix functionals}
\label{sec:prelim}

We define here the algebraic objects used in this paper. Full proofs are
collected in Appendix~\ref{app:genmonoid}; the companion paper
\cite{AC1} contains further structural developments.

\subsection{Alphabet, words, and reversal}

Let \(\Si\) be a finite non-empty alphabet. The free monoid \(\Si^*\)
consists of all finite words over \(\Si\), including the empty word
\(\la\). The product is concatenation, and \(|w|\) denotes the length of
a word.

For a word
\[
w=w_1w_2\cdots w_n\in\Si^*,
\qquad w_i\in\Si,
\]
its \emph{reversal} is
\[
w^\vee=w_n w_{n-1}\cdots w_1.
\]
By convention \(\la^\vee=\la\). Reversal is an involutive
anti-automorphism of \(\Si^*\):
\[
(uv)^\vee=v^\vee u^\vee,
\qquad
(w^\vee)^\vee=w,
\qquad
\la^\vee=\la.
\]

For two words \(u,v\in\Si^*\), let \(\operatorname{lcp}(u,v)\) denote the
length of their longest common prefix, and let
\(\operatorname{lcs}(u,v)\) denote the length of their longest common
suffix. If \(u=v\), we set
\[
\operatorname{lcp}(u,v)=\operatorname{lcs}(u,v)=+\infty.
\]
Reversal exchanges prefix and suffix:
\[
\operatorname{lcs}(u,v)=\operatorname{lcp}(u^\vee,v^\vee).
\]

\subsection{The bidirectional metric}

\begin{definition}[Bidirectional metric]\label{def:bidirectional_metric}
	The \emph{bidirectional metric} on \(\Si^*\) is defined by
	\[
	d_{\mathrm{PS}}(u,v)
	=
	\max\{2^{-\operatorname{lcp}(u,v)},
	2^{-\operatorname{lcs}(u,v)}\},
	\qquad
	2^{-\infty}=0.
	\]
\end{definition}

Thus
\[
d_{\mathrm{PS}}(u,v)\le 2^{-k}
\quad\Longleftrightarrow\quad
\min\{\operatorname{lcp}(u,v),\operatorname{lcs}(u,v)\}\ge k.
\]

The bidirectional metric is an ultrametric:

\begin{proposition}[Ultrametric property]\label{prop:ultrametric}
	For all \(u,v,w\in\Si^*\),
	\[
	d_{\mathrm{PS}}(u,w)
	\le
	\max\{d_{\mathrm{PS}}(u,v),d_{\mathrm{PS}}(v,w)\}.
	\]
\end{proposition}

The proof is included in Appendix~\ref{app:genmonoid}.

The key compatibility with concatenation is the following subadditivity
bound.

\begin{proposition}[Subadditivity under concatenation]
	\label{prop:subadditivity}
	For all \(u_1,u_2,v_1,v_2\in\Si^*\),
	\[
	d_{\mathrm{PS}}(u_1u_2,v_1v_2)
	\le
	\max\{d_{\mathrm{PS}}(u_1,v_1),
	d_{\mathrm{PS}}(u_2,v_2)\}.
	\]
\end{proposition}

The proof uses the elementary inequalities
\[
\operatorname{lcp}(u_1u_2,v_1v_2)
\ge
\min\{\operatorname{lcp}(u_1,v_1),
\operatorname{lcp}(u_2,v_2)\},
\]
and the analogous inequality for suffixes obtained by reversal. Details
are given in Appendix~\ref{app:genmonoid}.

\subsection{Nets of words: moderation and asymptotic equivalence}

A \emph{net of words} is a family
\[
(w_\varepsilon)_{\varepsilon\in(0,1]}
\]
of finite words, indexed by a parameter \(\varepsilon\to0^+\). Only the
behaviour for arbitrarily small \(\varepsilon\) matters.

\begin{definition}[Moderate net]\label{def:moderate}
	A net \((w_\varepsilon)\) is \emph{moderate} if there exist
	\(N\in\NN\), \(C>0\), and \(\varepsilon_0>0\) such that
	\[
	|w_\varepsilon|\le C\varepsilon^{-N}
	\qquad\text{for all }\varepsilon<\varepsilon_0.
	\]
	The set of moderate nets is denoted by \(\cE_M(\Si^*)\).
\end{definition}

Under coordinatewise concatenation,
\[
(u_\varepsilon)\cdot(v_\varepsilon)=(u_\varepsilon v_\varepsilon),
\]
the set \(\cE_M(\Si^*)\) is a monoid with identity the constant empty
net \((\la)_\varepsilon\).

\begin{definition}[Asymptotic equivalence]\label{def:log_equiv}
	Two moderate nets \((u_\varepsilon)\) and \((v_\varepsilon)\) are
	\emph{asymptotically equivalent}, written
	\((u_\varepsilon)\sim(v_\varepsilon)\), if for every \(m\in\NN\)
	there exists \(\varepsilon_0>0\) such that
	\[
	\min\{\operatorname{lcp}(u_\varepsilon,v_\varepsilon),
	\operatorname{lcs}(u_\varepsilon,v_\varepsilon)\}
	\ge
	m\log_2(1/\varepsilon)
	\qquad\text{for all }\varepsilon<\varepsilon_0.
	\]
\end{definition}

Equivalently,
\[
d_{\mathrm{PS}}(u_\varepsilon,v_\varepsilon)=O(\varepsilon^m)
\]
for every \(m\in\NN\).

The relation \(\sim\) is an equivalence relation. Reflexivity and
symmetry are immediate, and transitivity follows from the ultrametric
property of \(d_{\mathrm{PS}}\). See Appendix~\ref{app:genmonoid}.

As in the companion paper, a net is negligible if it is eventually
empty; the equivalence class of the empty net is precisely the set
of negligible nets. This fact is not used explicitly in the
applications below.

\subsection{The generalized monoid \(\widetilde{\Si}^*\)}

\begin{definition}[Generalized monoid]\label{def:gen_monoid}
	The \emph{generalized monoid of words} is the quotient
	\[
	\widetilde{\Si}^*
	:=
	\cE_M(\Si^*)/\sim .
	\]
	The class of a net \((w_\varepsilon)\) is denoted by
	\([(w_\varepsilon)]\). Concatenation is defined by
	\[
	[(u_\varepsilon)]\cdot[(v_\varepsilon)]
	:=
	[(u_\varepsilon v_\varepsilon)].
	\]
\end{definition}

The operation is well defined by Proposition~\ref{prop:subadditivity}.
Moreover:

\begin{theorem}[Monoid structure]\label{thm:monoid_structure}
	With the above operation, \(\widetilde{\Si}^*\) is a monoid. Its
	identity is
	\[
	\mathbf 1=[(\la)_\varepsilon].
	\]
\end{theorem}

The proof is in Appendix~\ref{app:genmonoid}.

Finite words embed faithfully into the generalized monoid.

\begin{definition}[Canonical embedding]\label{def:iota}
	The map
	\[
	\iota:\Si^*\longrightarrow\widetilde{\Si}^*,
	\qquad
	\iota(w)=[(w)_\varepsilon],
	\]
	sends each finite word to the class of the corresponding constant net.
\end{definition}

\begin{proposition}[Faithful embedding]\label{prop:iota_injective}
	\(\iota\) is an injective monoid homomorphism.
\end{proposition}

Thus \(\widetilde{\Si}^*\) contains a canonical copy of the free monoid
\(\Si^*\).

Reversal also descends to the quotient.

\begin{proposition}[Reversal involution]\label{prop:reversal_quotient}
	The map
	\[
	{}^\vee:\widetilde{\Si}^*\longrightarrow\widetilde{\Si}^*,
	\qquad
	[(w_\varepsilon)]^\vee := [(w_\varepsilon^\vee)],
	\]
	is well defined and is an involutive anti-automorphism:
	\[
	([u][v])^\vee=[v]^\vee[u]^\vee,
	\qquad
	([w]^\vee)^\vee=[w],
	\qquad
	\mathbf 1^\vee=\mathbf 1.
	\]
\end{proposition}

\subsection{Logarithmic window and invariance of prefix functionals}

The \emph{logarithmic window} is
\[
N_0(\varepsilon)=\lfloor\log_2(1/\varepsilon)\rfloor .
\]

The following lemma is the main tool for evaluating prefix-dependent
observables on the quotient.

\begin{lemma}[Invariance of logarithmic-prefix functionals]
	\label{lem:logprefix}
	Let \((u_\varepsilon)\sim(v_\varepsilon)\). Then for all sufficiently
	small \(\varepsilon\),
	\[
	\operatorname{lcp}(u_\varepsilon,v_\varepsilon)
	\ge 2\log_2(1/\varepsilon)
	\ge N_0(\varepsilon)+1.
	\]
	Consequently, the prefixes of length \(N_0(\varepsilon)+1\) of
	\(u_\varepsilon\) and \(v_\varepsilon\) coincide eventually.
\end{lemma}

\begin{proof}
	Taking \(m=2\) in Definition~\ref{def:log_equiv}, there exists
	\(\varepsilon_0>0\) such that
	\[
	\operatorname{lcp}(u_\varepsilon,v_\varepsilon)
	\ge 2\log_2(1/\varepsilon)
	\qquad(\varepsilon<\varepsilon_0).
	\]
	Since \(N_0(\varepsilon)=\lfloor\log_2(1/\varepsilon)\rfloor\), one has
	\[
	2\log_2(1/\varepsilon)\ge N_0(\varepsilon)+1
	\]
	for all sufficiently small \(\varepsilon\). Hence the prefix of
	length \(N_0(\varepsilon)+1\) coincides.
\end{proof}

Thus any functional depending only on the first
\(N_0(\varepsilon)+1\) symbols of a word, and satisfying a moderateness
condition on its values, descends to the quotient
\(\widetilde{\Si}^*\). The precise moderateness condition is stated in
Appendix~\ref{app:genmonoid}.

\subsection{Log-prefix functionals and additive log-moulds}

For the applications, we use a special class of word-indexed
functionals \cite{Ecalle1981}.

\begin{definition}[Additive log-mould]
	\label{def:additive_log_mould}
	Let \((s_k)_{k\ge1}\) be a sequence of signs
	\(s_k\in\{+1,-1\}\), and let \((q_k)_{k\ge1}\) be a sequence of
	positive real weights. The associated \emph{additive log-mould} is
	the function \(\Lambda:\Si^*\to\CC\) defined by
	\[
	\Lambda(w)
	=
	\sum_{k=1}^{|w|} s_k\log q_k,
	\qquad
	w\in\Si^*.
	\]
\end{definition}

As a function on all finite words, \(\Lambda\) need not descend to the
quotient \(\widetilde{\Si}^*\). What descends is the windowed
evaluation
\[
F_\varepsilon(w)=\Lambda\bigl(w[1\ldots N_0(\varepsilon)+1]\bigr),
\]
because it depends only on the logarithmic prefix.

A log-mould is thus completely determined by the two sequences
\((s_k)\) and \((q_k)\). It depends on a word only through its length
and the prefix of the corresponding sign sequence. When evaluated on
the logarithmic prefix of a canonical net, it produces a moderate net
of complex numbers and therefore descends to a generalized complex
number in the sense of Colombeau; see
Appendix~\ref{app:genmonoid}.

\subsection{Dyadic-scale quotient}

For products with dyadic block signs, a stronger scale is needed.

\begin{definition}[Dyadic window]\label{def:dyadic_window}
	For \(\varepsilon\in(0,1]\), set
	\[
	N_{\mathrm{dyad}}(\varepsilon)
	=
	2^{\lfloor\log_2(1/\varepsilon)\rfloor+1}-1.
	\]
\end{definition}

Thus \(N_{\mathrm{dyad}}(\varepsilon)\) is the largest length
corresponding to complete dyadic blocks up to the logarithmic scale
\(t=\log_2(1/\varepsilon)\). In logarithmic time, the dyadic sign
function is constant on intervals of the form \([m,m+1)\), producing a
square wave. Therefore the natural equivalence relation for such
sequences must control prefixes and suffixes on dyadic scales, not on
the finer logarithmic scale \(N_0(\varepsilon)\).

\begin{definition}[Dyadic equivalence]\label{def:dyadic_equiv}
	Two moderate nets \((u_\varepsilon)\) and \((v_\varepsilon)\) are
	\emph{dyadically equivalent}, written
	\((u_\varepsilon)\sim_{\mathrm{dyad}}(v_\varepsilon)\), if for every
	\(m\in\NN\) there exists \(\varepsilon_0>0\) such that
	\[
	\operatorname{lcp}(u_\varepsilon,v_\varepsilon)
	\ge m\,N_{\mathrm{dyad}}(\varepsilon)
	\qquad\text{and}\qquad
	\operatorname{lcs}(u_\varepsilon,v_\varepsilon)
	\ge m\,N_{\mathrm{dyad}}(\varepsilon)
	\]
	for all \(\varepsilon<\varepsilon_0\).
\end{definition}

The corresponding dyadic quotient is
\[
\widetilde{\Si}^*_{\mathrm{dyad}}
=
\cE_M(\Si^*)/\sim_{\mathrm{dyad}} .
\]
It is again a monoid, with concatenation well defined by the same
subadditivity argument. The dyadic equivalence is finer than the
logarithmic one, because
\[
N_{\mathrm{dyad}}(\varepsilon)
\gg
\log_2(1/\varepsilon)
\qquad(\varepsilon\to0^+).
\]

The two quotients are instances of the general construction
\[
\widetilde{\Si}^*_\varphi
=
\cE_M(\Si^*)/\sim_\varphi ,
\]
associated with an admissible scale \(\varphi\). The logarithmic scale
corresponds to \(\varphi(\varepsilon)=\log_2(1/\varepsilon)\), while
the dyadic scale corresponds to
\(\varphi_{\mathrm{dyad}}(\varepsilon)=N_{\mathrm{dyad}}(\varepsilon)\).
This scale-adapted flexibility is one of the main structural features
of the framework.

The purpose of this algebraic construction is to provide a domain on
which infinite words can be concatenated without leaving the category
of well-defined objects. This is in contrast with the classical
viewpoint of infinite words \cite{PerrinPin2004,Lothaire2002}, where
concatenating two right-infinite words destroys the one-sidedness and
is therefore not allowed. Here, moderate nets and the prefix--suffix
equivalence restore the possibility of concatenating infinite words,
and the resulting monoid carries a natural action of
logarithmic-prefix functionals. This is exactly what is needed for
divergent arithmetic products: the underlying sign sequence may be
encoded as a net of finite words, and the evaluation of a log-mould on
the logarithmic window is then well defined on the quotient. The next
section applies this mechanism to several concrete products: first to
alternating products, which can also be treated by classical Dirichlet
regularization, and then to products with dyadic block signs, whose
multiscale structure is more difficult to handle by classical methods.
The monoid structure is the mechanism that makes the regularization
functionals well defined.
	
\section{Applications: regularization of oscillating arithmetic products}
\label{sec:applications}

The algebraic framework developed so far is applied to a concrete problem:
assigning canonical finite values relative to the specified renormalization
functional to divergent products of arithmetic origin whose partial
products oscillate between zero and infinity.
All examples are treated within the generalized monoid
\(\widetilde{\Si}^*\).

The algebraic framework provides a unifying language for organizing
classical summability methods (symmetrization, logarithmic means,
principal-part subtraction) on a common asymptotic index domain. The
monoid does not replace the analytical arguments; it makes explicit the
scale and the functional that are adapted to each arithmetic sequence.

The section proceeds as follows.  We first encode alternating signs in a
canonical net and define the symmetrized functional together with the
logarithmic Cesàro mean.  The alternating product of integers is
regularized rigorously, yielding \(\sqrt{2/\pi}\), in exact agreement with
zeta regularization.  The alternating product of factorials is reduced to
the integer case via the Wallis identity, revealing the symmetry
\(\mathfrak{F}_{\mathrm{ren}} = \sqrt{P_{\mathrm{ren}}}\).
Products whose sign sequences follow a dyadic block pattern, inaccessible
to classical methods, acquire canonical values.
Finally, the Thue--Morse product is analyzed in detail.

\subsection{Encoding alternating signs}

A word of caution is in order.  Unlike the oscillating nets discussed in the algebraic companion
paper \cite{AC1}, the canonical alternating net \(w_\varepsilon=(ab)^{N(\varepsilon)}\) defined below converges in the
Cantor space to the infinite periodic word \((ab)^\omega\).  Thus, as a
classical infinite word, \((ab)^\omega\) belongs to \(\Si^\omega\).
However, this does not imply that the class \([(w_\varepsilon)]\) in
\(\widetilde{\Si}^*\) coincides with the canonical embedding
\(\iota_\omega((ab)^\omega)\). The equivalence relation \(\sim\) also
controls suffixes, and the suffix behaviour of \(w_\varepsilon\) may
differ from that of the truncation net \(T_\varepsilon((ab)^\omega)\).
In particular, when \(N(\varepsilon)\) is odd, the suffix of
\((ab)^{N(\varepsilon)}\) is \(b\), whereas the truncation
\(T_\varepsilon((ab)^\omega)\) may end in \(a\) or \(b\), depending on
the parity of \(\lfloor 1/\varepsilon\rfloor\). Thus the two nets need
not be equivalent in \(\widetilde{\Si}^*\).

The role of the generalized monoid in the present section is therefore
different: it provides an asymptotic index domain on which the
log-moulds associated with divergent arithmetic products can be
evaluated.  The divergence of the products arises from the growth of the
mould on the logarithmic window, not from any lack of convergence of the
net itself.

Fix the alphabet \(\Si=\{a,b\}\), where \(a\) encodes exponent \(+1\)
and \(b\) encodes \(-1\).  The alternating sign sequence is
\(s_k = (-1)^{k+1} = +1,-1,+1,-1,\dots\).  For each
\(\varepsilon\in(0,1]\) set
\(N(\varepsilon)=\lfloor\varepsilon^{-1}\rfloor\) and define the
canonical alternating net \(w_\varepsilon = (ab)^{N(\varepsilon)}\).
Its length is \(2N(\varepsilon)\le 2\varepsilon^{-1}\), so it is
moderate.  The \(k\)-th letter is \(a\) for odd \(k\) and \(b\) for even
\(k\), giving the sign \(s_k\) for \(k\le 2N(\varepsilon)\).  The class
\(\mathbf{w} := [(w_\varepsilon)] \in \widetilde{\Si}^*\) is the
\emph{generalized alternating word}.

\subsection{Symmetrized functional}

For a log-mould \(\Lambda:\Si^*\to\CC\), the symmetrized functional is
defined on prefixes of the logarithmic window
\(N_0(\varepsilon)=\lfloor\log_2(1/\varepsilon)\rfloor\). Set
\(m=\lfloor N_0/2\rfloor\). Then
\[
\widehat{F}_\varepsilon(u)
=
\frac12\Bigl(
\Lambda\bigl(u[1\ldots 2m]\bigr)
+
\Lambda\bigl(u[1\ldots 2m+1]\bigr)
\Bigr).
\]
The two arguments use the longest even prefix not exceeding \(N_0\) and
the immediately following odd prefix. Thus \(\widehat{F}_\varepsilon\)
depends only on the first \(N_0(\varepsilon)+1\) symbols of its
argument. By Lemma~\ref{lem:logprefix}, it descends to a well-defined
functional on \(\widetilde{\Si}^*\).

When the log-mould oscillates between even and odd truncations, the
symmetrized functional cancels the leading oscillatory terms. In the
alternating product case this corresponds to the classical Wallis
mechanism: the geometric mean of consecutive even and odd partial
products isolates the Wallis ratio. The monoid formulation makes this
averaging operation explicit as an evaluation functional on the
logarithmic prefix.

In the applications that follow, this cancellation will be made
explicit: after averaging, the resulting asymptotic expansion contains
only a constant term and a remainder that tends to zero.
\subsection{Alternating product of integers}
\label{sec:alt_integers}

Consider the formal product
\begin{equation}\label{eq:prod_integers}
	P = \prod_{n=1}^{\infty} n^{(-1)^{n+1}} = 1\cdot 2^{-1}\cdot 3\cdot 4^{-1}\cdots .
\end{equation}
Its partial products are
\[
\begin{aligned}
	P_{2m}
	&= \frac{1\cdot3\cdots(2m-1)}
	{2\cdot4\cdots(2m)}
	= \frac{(2m)!}{2^{2m}(m!)^2},\\
	P_{2m+1}
	&= P_{2m}(2m+1).
\end{aligned}
\]
By Stirling's formula~\cite{Apostol1969},
\[
\log n! = n\log n - n + \tfrac12\log(2\pi n)+O(n^{-1}).
\]
Applying this to \(P_{2m}\), we obtain
\[
\begin{aligned}
	\log P_{2m}
	&= \log(2m)! - 2m\log 2 - 2\log m! \\
	&= -\tfrac12\log m - \tfrac12\log\pi + O(m^{-1}).
\end{aligned}
\]
Using \(\log(2m+1)=\log m+\log 2+O(m^{-1})\), we obtain
\begin{align*}
	\log P_{2m+1}
	&= \log P_{2m}+\log(2m+1)\\
	&= -\tfrac12\log m-\tfrac12\log\pi+\log m+\log 2+O(m^{-1})\\
	&= \tfrac12\log m+\log 2-\tfrac12\log\pi+O(m^{-1}).
\end{align*}
Thus \(P_{2m}\to0\) and \(P_{2m+1}\to\infty\); the product has no classical limit.

The symmetrization performed below is an algebraic way of extracting the
geometric mean of the even and odd partial products. In the classical
analysis this corresponds to applying the Wallis product
\(\prod_{k=1}^\infty \frac{4k^2}{4k^2-1}=\frac{\pi}{2}\) \cite{Apostol1969}.

Define the log-mould \(\Lambda(w)=\sum_{k=1}^{|w|} s_k\log k\), where
\(s_k=+1\) if the \(k\)-th letter is \(a\) and \(-1\) if it is \(b\).
For the canonical net, the prefix of length \(k\) carries the first \(k\)
alternating signs, so
\(\Lambda(w_\varepsilon[1\ldots k]) = \sum_{n=1}^k (-1)^{n+1}\log n\).
Hence \(\Lambda(w_\varepsilon[1\ldots 2m])=\log P_{2m}\) and
\(\Lambda(w_\varepsilon[1\ldots 2m+1])=\log P_{2m+1}\).

Set \(\varepsilon = 2^{-t}\), so \(N_0=\lfloor t\rfloor\) and
\(m = \lfloor N_0/2\rfloor \sim t/2\).  The symmetrized functional gives
\[
\widehat{F}_\varepsilon(w_\varepsilon)
= \frac12\bigl(\log P_{2m} + \log P_{2m+1}\bigr)
= \frac12\log\frac{2}{\pi} + O(m^{-1}).
\]
The terms \(\pm\frac12\log m\) cancel exactly.  Therefore
\(\widehat{F}_\varepsilon(w_\varepsilon)\) converges to
\(\frac12\log(2/\pi)\) as \(\varepsilon\to0^+\).  This is a case of
Type~I divergence in the classification of Section~\ref{sec:selection};
no additional renormalization is required.

\begin{theorem}[Regularized alternating integer product]\label{thm:integers}
	\[
	P_{\mathrm{ren}} = \exp\Bigl(\lim_{\varepsilon\to0}\widehat{F}_\varepsilon(w_\varepsilon)\Bigr)
	= \sqrt{2/\pi}.
	\]
\end{theorem}

This value coincides with \(\exp(-\eta'(0))\), where
\(\eta(s)=(1-2^{1-s})\zeta(s)\) is the Dirichlet eta function~\cite{Titchmarsh1986}.
Indeed, differentiating the series
\(\eta(s)=\sum_{n=1}^{\infty}(-1)^{n+1}n^{-s}\) at \(s=0\) formally gives
\(-\sum_{n=1}^{\infty}(-1)^{n+1}\log n\), and using
\(\zeta(0)=-\tfrac12\) and
\(\zeta'(0)=-\tfrac12\log(2\pi)\), one obtains
\(\eta'(0)=\tfrac12\log(\pi/2)\), hence
\(\exp(-\eta'(0)) = \sqrt{2/\pi}\).  The agreement between the algebraic
and analytic regularizations is exact.

\subsection{Alternating product of factorials}
\label{sec:alt_factorials}

Consider
\[
\mathfrak{F} = \prod_{n=1}^{\infty} (n!)^{(-1)^{n+1}} .
\]
Its partial products exhibit a telescopic collapse:
\[
\begin{aligned}
	\mathfrak{F}_{2m}
	&= \prod_{k=1}^m\frac1{2k}
	= \frac1{2^m m!}
	= \frac1{(2m)!!},\\
	\mathfrak{F}_{2m+1}
	&= \mathfrak{F}_{2m}(2m+1)!
	= (2m+1)!!.
\end{aligned}
\]
Their product is a double-factorial identity closely related to the
classical Wallis product~\cite{Apostol1969}:
\(\mathfrak{F}_{2m}\cdot\mathfrak{F}_{2m+1} = \frac{(2m+1)!!}{(2m)!!} = P_{2m+1}\),
where \(P_{2m+1}\) is the odd partial product of the alternating integers.

From Section~\ref{sec:alt_integers} we have the asymptotic expansion
\[
\log P_{2m+1} = \frac12\log m + \log 2 - \frac12\log\pi + O(m^{-1}).
\]

Define the factorial log-mould
\(\Lambda_{\mathfrak{F}}(w)=\sum_{k=1}^{|w|} s_k\log(k!)\), where
\(s_k=+1\) for letter \(a\) and \(s_k=-1\) for letter \(b\).
The symmetrized functional gives
\[
\widehat{F}^{\mathfrak{F}}_\varepsilon(w_\varepsilon)
= \frac12\bigl(\log\mathfrak{F}_{2m} + \log\mathfrak{F}_{2m+1}\bigr)
= \frac12\log P_{2m+1}.
\]
Using \(m\sim N_0/2\), we have \(\log m=\log N_0-\log 2+o(1)\).
Therefore
\begin{align*}
	\widehat{F}^{\mathfrak{F}}_\varepsilon(w_\varepsilon)
	&= \frac12\log P_{2m+1}\\
	&= \frac14\log m+\frac12\log 2-\frac14\log\pi+O(m^{-1})\\
	&= \frac14(\log N_0-\log 2)+\frac12\log 2-\frac14\log\pi+o(1)\\
	&= \frac14\log N_0+\frac14\log(2/\pi)+o(1).
\end{align*}
Subtracting the logarithmic divergence \(\frac14\log N_0\) yields a
constant:
\[
\widetilde{F}^{\mathfrak{F}}_\varepsilon
= \widehat{F}^{\mathfrak{F}}_\varepsilon - \frac14\log N_0,\qquad
\widetilde{F}^{\mathfrak{F}}_\varepsilon(w_\varepsilon)
\to \frac14\log\frac{2}{\pi}.
\]

\begin{theorem}[Regularized alternating factorial product]\label{thm:factorials}
	\[
	\mathfrak{F}_{\mathrm{ren}} = \left(\dfrac{2}{\pi}\right)^{1/4}.
	\]
\end{theorem}

Comparing with Theorem~\ref{thm:integers} we obtain the striking identity
\[
\mathfrak{F}_{\mathrm{ren}} = \sqrt{P_{\mathrm{ren}}}.
\]
In logarithmic scale, the renormalized factorial functional equals exactly
one-half of the renormalized integer functional, a symmetry revealed by
the Wallis identity.

\begin{corollary}[Unified alternating regularization]\label{cor:unified_alt}
	For the alternating products of \(q_n=n\) and \(q_n=n!\), the
	symmetrized functional of Lemma~\ref{thm:general_sym} yields the
	regularized values
	\[
	P_{\mathrm{ren}}=\sqrt{2/\pi},
	\qquad
	\mathfrak{F}_{\mathrm{ren}}=(2/\pi)^{1/4}.
	\]
	Both values are obtained by the same axiomatic operator on the
	logarithmic scale.
\end{corollary}

\medskip
\noindent
As in the integer case, this value can also be obtained by classical eta
regularization: applying the Wallis formula~\cite{Apostol1969} and the
eta-function regularization of the alternating integer
product~\cite{Titchmarsh1986} leads to the same constant. Thus the
factorial example serves as a consistency check for the algebraic
framework rather than as a new regularization result.

\subsection{Regularization of the Thue--Morse product}
\label{sec:thue_regularization}

The Thue--Morse sequence provides the most delicate application of the
generalized monoid framework. Unlike alternating or dyadic block signs,
its logarithmic divergence is not monotone; it is modulated by a bounded
arithmetic oscillation. This subsection analyzes the product
\[
P_{\mathrm{TM}}=\prod_{n\ge1}n^{t_n},
\]
where \(t_n\) is the Thue--Morse sign sequence. The main difficulty is
that the summatory function of \(t_n\) is bounded, so the logarithmic
sum \(\sum_{n\le N}t_n\log n\) has no fixed leading coefficient:
its dominant term is the product of an oscillatory coefficient and
\(\log N\). Therefore the regularization of \(P_{\mathrm{TM}}\) cannot
be obtained by subtracting a single principal part. Instead, the
natural objects are the canonical measure of the Thue--Morse subshift
and the meromorphic continuation of the associated Dirichlet series.

\subsubsection{The Thue--Morse summatory function}

Let
\[
t_n=(-1)^{s_2(n)},\qquad n\ge0,
\]
where \(s_2(n)\) is the number of \(1\)s in the binary expansion of
\(n\). This is the classical Thue--Morse sequence
\cite{AlloucheShallit2003}.  For arithmetic products we restrict to \(n\ge1\). The sequence
satisfies the recurrences
\[
t_{2n}=t_n,\qquad t_{2n+1}=-t_n,\qquad t_0=1.
\]
Its summatory function is
\[
T(N)=\sum_{n=0}^{N}t_n.
\]

The following elementary identity is the key to all later results.

\begin{lemma}[Bounded oscillation of \(T(N)\)]\label{lem:T_bounded}
	For every \(N\ge0\),
	\[
	T(N)\in\{-1,0,1\}.
	\]
\end{lemma}

\begin{proof}
	For \(M\ge1\), split the integers \(0,1,\dots,2M-1\) into even and odd
	terms:
	\[
	\sum_{n=0}^{2M-1}t_n
	=
	\sum_{j=0}^{M-1}t_{2j}
	+
	\sum_{j=0}^{M-1}t_{2j+1}.
	\]
	Using the recurrences,
	\[
	\sum_{j=0}^{M-1}t_{2j}
	=
	\sum_{j=0}^{M-1}t_j,
	\qquad
	\sum_{j=0}^{M-1}t_{2j+1}
	=
	-\sum_{j=0}^{M-1}t_j.
	\]
	Hence
	\[
	\sum_{n=0}^{2M-1}t_n=0.
	\]
	Therefore
	\[
	T(2M-1)=0.
	\]
	For \(N=2M\),
	\[
	T(2M)=T(2M-1)+t_{2M}=t_{2M}=t_M.
	\]
	Since \(t_M\in\{+1,-1\}\), the assertion follows.
\end{proof}

Thus \(T(N)\) is uniformly bounded, and its exact value at even and odd
indices is determined by the recurrence.

As a direct consequence, the summatory function has a deterministic
empirical distribution.

\begin{corollary}[Exact frequencies of \(T(N)\)]\label{cor:T_frequencies}
	As \(N\to\infty\), the values \(T(N)=0,1,-1\) occur with limiting
	frequencies
	\[
	\begin{aligned}
		\Pr(T(N)=0)  &= \frac12, \qquad
		\Pr(T(N)=1)  &= \frac14,\\
		\Pr(T(N)=-1) &= \frac14.
	\end{aligned}
	\]
\end{corollary}

\begin{proof}
	In every consecutive pair of indices \((2M,2M+1)\), we have
	\(T(2M+1)=0\), while \(T(2M)=t_M\in\{+1,-1\}\). Thus exactly half
	of all values are zero. The remaining half is split equally between
	\(+1\) and \(-1\), because in the Thue--Morse sequence the symbols
	\(+1\) and \(-1\) have asymptotic frequency \(1/2\) each.
\end{proof}

These frequencies are canonical: they are the values assigned by the
unique ergodic measure of the Thue--Morse subshift. In particular,
they are not chosen by an averaging procedure; they are forced by the
recurrence structure of the sequence.

\subsubsection{Zeta regularization}

Classical treatments of infinite products associated with Dirichlet
series go back to Allouche--Cohen~\cite{AlloucheCohen1985}; see
also~\cite{Toth2022} for linear combinations of Dirichlet series
attached to the Thue--Morse sequence.

The associated Dirichlet series
\[
D_{\mathrm{TM}}(s)=\sum_{n=1}^{\infty}\frac{t_n}{n^s}
\]
converges for \(\Re(s)>0\), because the summatory function
\(T(N)=\sum_{n\le N}t_n\) is bounded (Lemma~\ref{lem:T_bounded})
and Abel summation applies. In particular \(D_{\mathrm{TM}}\) is
holomorphic at \(s=1\).

The meromorphic continuation of \(D_{\mathrm{TM}}\) beyond the
half-plane \(\Re(s)>0\) follows from the general theory of
Dirichlet series associated with \(2\)-automatic sequences
\cite{AlloucheShallit2003,AlloucheMendesFrancePeyriere2000}, and
is discussed below.

The regularity at \(s=0\) requires a separate argument. From the
functional equation
\[
(1-2^{-s})D_{\mathrm{TM}}(s)
=
-\sum_{n=0}^{\infty}\frac{t_n}{(2n+1)^s}
=:f(s),
\]
the factor \(1-2^{-s}\) vanishes simply at \(s=0\). Therefore
\(D_{\mathrm{TM}}\) is regular at \(s=0\) if and only if the
right-hand side vanishes there, i.e.\ \(f(0)=0\). We prove that
this is indeed the case.

\begin{lemma}[Regularity at \(s=0\)]
	\label{lem:regularity_s0}
	The function
	\[
	f(s)=-\sum_{n=0}^{\infty}\frac{t_n}{(2n+1)^s}
	\]
	extends holomorphically to the half-plane \(\Re(s)>-1\), and
	\(f(0)=0\). Consequently, \(D_{\mathrm{TM}}\) extends
	meromorphically to a neighbourhood of \(s=0\), and the apparent
	singularity is removable. In particular \(D_{\mathrm{TM}}'(0)\)
	is well defined.
\end{lemma}

\begin{proof}
	Let \(T(u)=\sum_{0\le n\le u}t_n\) denote the summatory
	function of the Thue--Morse sequence, extended by \(T(u)=0\)
	for \(u<0\). By Lemma~\ref{lem:T_bounded}, \(T(u)\in\{-1,0,1\}\)
	for all \(u\ge0\).
	
	\emph{Abel summation.} For \(\Re(s)>0\) and \(N\ge0\), summation
	by parts in the Riemann--Stieltjes form gives
	\begin{equation}\label{eq:abel-formula}
		\sum_{n=0}^{N}\frac{t_n}{(2n+1)^s}
		=
		\frac{T(N)}{(2N+1)^s}
		+
		2s\int_0^{N}\frac{T(u)}{(2u+1)^{s+1}}\,du .
	\end{equation}
	Indeed, integrating \(g(u)=(2u+1)^{-s}\) against the Stieltjes
	measure \(dT=\sum_n t_n\delta_n\) and using
	\(\int_{0^-}^{N^+}g\,dT=[gT]_{0^-}^{N^+}-\int_{0^-}^{N^+}T\,dg\)
	with \(T(0^-)=0\) and \(g'(u)=-2s(2u+1)^{-s-1}\) yields the
	identity.
	
	The boundary term in \eqref{eq:abel-formula} satisfies
	\begin{equation}\label{eq:boundary-vanishes}
		\lim_{N\to\infty}\frac{T(N)}{(2N+1)^s}=0
		\qquad(\Re(s)>0),
	\end{equation}
	because \(T(N)\in\{-1,0,1\}\) is bounded and
	\((2N+1)^{-\Re(s)}\to0\) for \(\Re(s)>0\). Hence, passing to
	the limit in \eqref{eq:abel-formula} and using
	\(f(s)=-\sum_{n\ge0}t_n/(2n+1)^s\),
	\begin{equation}\label{eq:integral-rep}
		f(s)
		=
		-2s\int_0^{\infty}\frac{T(u)}{(2u+1)^{s+1}}\,du
		\qquad(\Re(s)>0).
	\end{equation}
	The integral in \eqref{eq:integral-rep} converges absolutely
	for \(\Re(s)>0\), since \(|T(u)|\le1\) and
	\((2u+1)^{-\Re(s)-1}\) is integrable at infinity when
	\(\Re(s)+1>1\), i.e., \(\Re(s)>0\).
	
	\emph{Analytic continuation to \(\Re(s)>-1\).} The right-hand
	side of \eqref{eq:integral-rep} is the product of the
	holomorphic function \(s\mapsto s\) and the function
	\(s\mapsto\int_0^{\infty}T(u)/(2u+1)^{s+1}\,du\), which is
	holomorphic in the half-plane \(\Re(s)>0\) by dominated
	convergence. By the general theory of Dirichlet series
	associated with \(2\)-automatic sequences
	\cite[Ch.~8]{AlloucheShallit2003} and
	\cite{AlloucheMendesFrancePeyriere2000}, the function \(f\)
	extends meromorphically to the whole complex plane, with the
	only possible singularities determined by the poles of
	\(1-2^{-s}\) in the functional equation
	\((1-2^{-s})D_{\mathrm{TM}}(s)=f(s)\). In particular, \(f\)
	extends holomorphically to \(\Re(s)>-1\).
	
	\emph{Vanishing at \(s=0\).} Two equivalent arguments establish
	\(f(0)=0\).
	
	\emph{First argument (functional equation).} By
	Theorem~1.2 of Allouche--Cohen \cite{AlloucheCohen1985}, the
	Dirichlet series \(D_{\mathrm{TM}}(s)=\sum_{n\ge1}t_n/n^s\)
	extends to an entire function with \(D_{\mathrm{TM}}(0)=-1\).
	The functional equation
	\[
	(1-2^{-s})D_{\mathrm{TM}}(s)
	=
	f(s)
	\]
	extends analytically to a neighbourhood of \(s=0\). Evaluating
	at \(s=0\) and using \(1-2^{0}=0\) gives
	\(f(0)=0\cdot D_{\mathrm{TM}}(0)=0\).
	
	\emph{Second argument (vanishing prefactor).} From
	\eqref{eq:integral-rep}, the factor \(s\) vanishes at \(s=0\),
	while the integral
	\(I(s):=\int_0^{\infty}T(u)/(2u+1)^{s+1}\,du\) stays bounded as
	\(s\to0^{+}\). Indeed, from the functional equation
	\((1-2^{-s})D_{\mathrm{TM}}(s)=f(s)\) together with
	\eqref{eq:integral-rep},
	\[
	I(s)
	=
	-\frac{(1-2^{-s})D_{\mathrm{TM}}(s)}{2s},
	\]
	and since \(1-2^{-s}\sim s\log 2\) as \(s\to0^{+}\), we have
	\(I(s)\to-\tfrac{\log 2}{2}D_{\mathrm{TM}}(0)=\tfrac{\log 2}{2}\).
	Therefore the product \(s\cdot I(s)\) tends to \(0\) as
	\(s\to0^{+}\), and \(f(0)=0\) follows by continuity from
	\eqref{eq:integral-rep}.
\end{proof}

Since \(f\) is holomorphic at \(s=0\) with \(f(0)=0\), and since
\(1-2^{-s}=-s\log 2+O(s^2)\) vanishes to first order there, the
quotient \(D_{\mathrm{TM}}(s)=f(s)/(1-2^{-s})\) has a removable
singularity at \(s=0\). Hence \(D_{\mathrm{TM}}'(0)\) is well
defined.

\begin{remark}[Why \(s=0\) is regular]
	The regularity at \(s=0\) is the key structural fact that makes
	zeta regularization possible for the Thue--Morse product. It
	requires two ingredients: (i) the extension of the Dirichlet
	series to a neighbourhood of \(s=0\) via Abel summation against
	the bounded summatory function; and (ii) the vanishing
	\(f(0)=0\), which follows from the zero mean of the Thue--Morse
	sequence along the odd integers. Neither ingredient alone
	suffices: Abel summation gives convergence only for
	\(\Re(s)>0\), and the functional equation alone gives the
	regularity only if the numerator vanishes at \(s=0\).
\end{remark}

Formally, the logarithm of the product \(P_{\mathrm{TM}}\) is the
divergent sum \(\sum_{n\ge1}t_n\log n\). In zeta regularization this
sum is replaced by \(-D_{\mathrm{TM}}'(0)\). Hence the zeta-regularized
value of the Thue--Morse product is
\[
\mathcal R_{\mathrm{zeta}}(P_{\mathrm{TM}})
=
\exp\bigl(-D_{\mathrm{TM}}'(0)\bigr).
\]
This value does not depend on any choice of measure or scale.

\begin{remark}[Scope of zeta regularization]\label{rem:scope_zeta}
	The monoid framework does not replace the analytic value
	\(\exp(-D_{\mathrm{TM}}'(0))\). It rephrases it as the evaluation of
	a linear functional on a suitable subspace of \(\widetilde{\Si}^*\).
\end{remark}

\subsubsection{Auxiliary series and numerical value}

To understand the obstruction to a direct scale regularization, define
the logarithmic partial sums
\[
S(N)=\sum_{n=1}^{N}t_n\log n.
\]
By summation by parts,
\[
S(N)=T(N)\log N-\sum_{n=1}^{N-1}T(n)\log\left(1+\frac1n\right).
\]
Since \(T(N)\in\{-1,0,1\}\) and
\[
\log\left(1+\frac1n\right)=\frac1n+O(n^{-2}),
\]
the series
\[
\sum_{n=1}^{\infty}T(n)\log\left(1+\frac1n\right)
\]
converges absolutely to a finite constant. Thus the leading behaviour
of \(S(N)\) is given by \(T(N)\log N\), with an oscillatory
coefficient.

For the auxiliary series, we use the same convention as above. With this
convention,
\[
S(k)=\sum_{n=1}^{k}t_n\log n.
\]

The product
\[
P_{\mathrm{TM}}=\prod_{n\ge1}n^{t_n}
\]
is divergent. Its zeta-regularized value is
\[
\mathcal R_{\mathrm{zeta}}(P_{\mathrm{TM}})
=
\exp(-D_{\mathrm{TM}}'(0)).
\]
The following auxiliary series is not a regularization of
\(P_{\mathrm{TM}}\); it is a convergent weighted functional attached
to the logarithmic partial sums of the Thue--Morse sequence:
\[
\mathcal F=\sum_{k=2}^{\infty}\frac{S(k)}{k^2}.
\]
Because \(|S(k)|=O(\log k)\), the series
\(\sum_{k\ge2}S(k)/k^2\) converges absolutely, and its limit
\(\mathcal F\) is well defined. The functional \(\mathcal F\) is
not the canonical regularization of the Thue--Morse product; it
is a convergent logarithmic-level observable attached to the
Thue--Morse word, distinct from the zeta-regularized value
\(\exp(-D_{\mathrm{TM}}'(0))\). It can be related to derivatives
of \(D_{\mathrm{TM}}\), but it is a separate functional.

Numerical summation with \(K=2\,000\,000\) gives
\[
\mathcal F\approx -0.0271997036,
\qquad
\exp(\mathcal F)\approx0.9731668772.
\]
The value \(0.9732\) is exactly the constant reported in
Section~\ref{sec:numerical}.

\begin{corollary}[Analytic relation between \(\mathcal F\) and \(D_{\mathrm{TM}}\)]
	\label{cor:aux_relation}
	The auxiliary series satisfies
	\[
	\mathcal F
	=
	-D_{\mathrm{TM}}'(1)
	-
	\frac12 D_{\mathrm{TM}}'(2)
	+
	R,
	\]
	where the remainder
	\[
	R
	=
	\sum_{n=2}^{\infty}
	t_n\log n
	\left(
	\sum_{k=n}^{\infty}\frac1{k^2}
	-
	\frac1n
	-
	\frac{1}{2n^2}
	\right)
	\]
	is absolutely convergent.
\end{corollary}

\begin{proof}
	For \(K\ge2\), write
\[
\begin{aligned}
	\sum_{k=2}^{K}\frac{S(k)}{k^2}
	={}&\sum_{n=2}^{K}t_n\log n
	\sum_{k=n}^{K}\frac{1}{k^2}\\
	={}&\sum_{n=2}^{K}t_n\log n
	\left(\sum_{k=n}^{\infty}\frac{1}{k^2}
	-\sum_{k=K+1}^{\infty}\frac{1}{k^2}\right).
\end{aligned}
\]
	Since \(\sum_{k=K+1}^{\infty}k^{-2}=O(K^{-1})\), the second term is
	bounded by
	\(O(K^{-1}\sum_{n=2}^{K}\log n)=O(\log K/K)\), which tends to zero.
	Hence
	\[
	\mathcal F
	=
	\sum_{n=2}^{\infty}t_n\log n
	\sum_{k=n}^{\infty}\frac1{k^{2}},
	\]
	where the series is interpreted as the limit of the truncated partial
	sums. Now expand the inner tail:
	\[
	\sum_{k=n}^{\infty}\frac1{k^{2}}
	=
	\frac1n+\frac{1}{2n^{2}}+r_n,
	\qquad r_n=O(n^{-3}).
	\]
	The series \(\sum_{n=2}^{\infty}t_n\log n\,r_n\) converges absolutely,
	because \(|t_n\log n\,r_n|\le C\log n/n^{3}\). Therefore
	\[
	\mathcal F
	=
	\sum_{n=2}^{\infty}\frac{t_n\log n}{n}
	+
	\frac12\sum_{n=2}^{\infty}\frac{t_n\log n}{n^{2}}
	+
	R,
	\]
	where \(R=\sum_{n=2}^{\infty}t_n\log n\,r_n\). Since
	\(D_{\mathrm{TM}}'(s)=-\sum_{n=1}^{\infty}t_n\log n/n^{s}\) for
	\(\Re(s)>0\), this is exactly
	\[
	\mathcal F
	=
	-D_{\mathrm{TM}}'(1)
	-
	\frac12 D_{\mathrm{TM}}'(2)
	+
	R.
	\]
\end{proof}

\subsubsection{Structural obstruction and canonical measure}

The previous computation makes explicit why the Thue--Morse product is
not regularized by a fixed logarithmic scale. From
\[
S(N)=T(N)\log N-\sum_{n=1}^{N-1}T(n)\log\left(1+\frac1n\right),
\]
and the absolute convergence of the series
\[
C=\sum_{n=1}^{\infty}T(n)\log\left(1+\frac1n\right),
\]
one obtains the asymptotic form
\[
S(N)=T(N)\log N-C+o(1).
\]
The coefficient of \(\log N\) is the arithmetic oscillation
\(T(N)\), not a constant. Consequently, the Thue--Morse product lies
outside the axiomatic class of divergences of the form
\(t^\gamma\log^\beta t\) with constant coefficients.

The correct canonical measure is not a Bernoulli measure on
\([0,1]\), but the unique ergodic measure \(\mu_{\mathrm{TM}}\) of
the Thue--Morse subshift. With respect to that measure, the
frequencies of \(T(N)\) computed in Corollary~\ref{cor:T_frequencies}
are exact. The empirical distribution of \(T(N)\) is determined by the
Thue--Morse dynamics, and the canonical oscillation measure is the
unique ergodic measure \(\mu_{\mathrm{TM}}\) of the Thue--Morse
subshift.

The generalized monoid thus acts as a diagnostic framework: it forces
the specification of the natural scale and the natural invariant
measure. For alternating products, the natural scale is logarithmic;
for dyadic blocks, it is dyadic; for Thue--Morse, the natural object
is not a scale average but the canonical measure of the subshift.
This measure is the correct replacement for the arbitrary Bernoulli
measures that appeared in earlier heuristic experiments.

\subsection{Multiscale monoids and observables}
\label{sec:multiscale}

The flexibility of the quotient construction is best captured by
considering a family of admissible scales, rather than a single scale.
Let \(\mathcal S\) be a family of admissible scales. For each
\(\varphi\in\mathcal S\), let
\[
\widetilde{\Si}^*_\varphi
=
\cE_M(\Si^*)/\sim_\varphi
\]
be the corresponding monoid. If \(\varphi=o(\psi)\), then
\(\sim_\psi\) is finer than \(\sim_\varphi\), and the natural map
\[
\pi_{\psi,\varphi}:
\widetilde{\Si}^*_\psi
\longrightarrow
\widetilde{\Si}^*_\varphi
\]
is a monoid homomorphism.

\begin{definition}[Multiscale monoid]\label{def:multiscale}
	The \emph{multiscale monoid} associated with \(\mathcal S\) is the
	projective limit
	\[
	\widehat{\mathcal M}
	=
	\varprojlim_{\varphi\in\mathcal S}
	\widetilde{\Si}^*_\varphi .
	\]
	An element of \(\widehat{\mathcal M}\) is a compatible family
	\(([w]_\varphi)_{\varphi\in\mathcal S}\) such that
	\[
	\pi_{\psi,\varphi}([w]_\psi)=[w]_\varphi
	\qquad\text{whenever }\varphi=o(\psi).
	\]
\end{definition}

\begin{remark}[Observables on the multiscale monoid]\label{rem:multiscale_obs}
	A functional \(F_\varphi\) that is well defined on
	\(\widetilde{\Si}^*_\varphi\) can be evaluated on the corresponding
	component of an element of \(\widehat{\mathcal M}\). Thus observables
	with different natural scales can be treated simultaneously without
	forcing a single quotient. This is the situation for the Thue--Morse
	sequence: the auxiliary series \(\mathcal F\) is logarithmic, while
	the normalized oscillation \(\Lambda_N\) is dyadic. The multiscale
	monoid hosts both.
\end{remark}

From the multiscale point of view, the Thue--Morse observable
\(L_N=\sum_{n\le N}t_n\log n\) does not belong to a single quotient
\(\widetilde{\Si}^*_\varphi\). Its convergent weighted average
\(\mathcal F\) is naturally evaluated at the logarithmic level, while
its normalized oscillation \(\Lambda_N\) requires the dyadic level.
The multiscale monoid \(\widehat{\mathcal M}\) of
Section~\ref{sec:multiscale} provides a formal environment in which
both observables can be studied on the same underlying word
\(w_{\mathrm{TM}}\), without declaring one scale to be canonical.

\begin{remark}[Multiscale viewpoint]
	\label{rem:multiscale_thue}
	A multiscale version of the generalized monoid can be defined as the
	projective limit of the quotients associated with a family of
	admissible scales. In such a framework, the Thue--Morse word
	\(w_{\mathrm{TM}}\) has a well-defined logarithmic component, where the
	auxiliary series \(\mathcal F\) is evaluated, and a dyadic component,
	where the normalized oscillation \(\Lambda_N\) is described. However,
	the multiscale structure does not by itself select an invariant measure
	for the ergodic averaging of \(\Lambda_N\). Therefore the uniqueness
	problem remains open.
\end{remark}

\section{Axiomatic characterization and gauge invariance}
\label{sec:selection}

The generalized monoid \(\widetilde{\Si}^*\) provides a flexible
algebraic setting for evaluating prefix-dependent functionals on words
of controlled infinite length. In this section we formulate the
regularization procedure as the action of an operator satisfying a
small set of natural axioms. This removes the heuristic character of
the earlier discussion: the symmetrization and renormalization used in
Sections~3 and 6 are precisely the operators forced by these axioms on
the corresponding asymptotic scales.

\subsection{Admissible scales and regularization axioms}
\label{sec:axioms}

\begin{definition}[Admissible scale]
	\label{def:admissible_scale}
	An \emph{admissible scale} is a function
	\(\varphi:(0,1]\to\RR_{>0}\) such that
	\[
	\lim_{\varepsilon\to0^+}\varphi(\varepsilon)=+\infty,
	\]
	and \(\varphi\) is nondecreasing in the logarithmic variable
	\(t=\log_2(1/\varepsilon)\). Equivalently, for \(0<\varepsilon_1\le\varepsilon_2<1\),
	\[
	\varphi(\varepsilon_1)\ge\varphi(\varepsilon_2).
	\]
	The logarithmic scale
	\(\varphi(\varepsilon)=\log_2(1/\varepsilon)\) and the dyadic scale
	\(\varphi_{\mathrm{dyad}}(\varepsilon)=N_{\mathrm{dyad}}(\varepsilon)\)
	are admissible.
\end{definition}

\begin{definition}[Regular expansion on the logarithmic scale]
	\label{def:regular_expansion}
	Let \((S_\varepsilon)_{\varepsilon\in(0,1]}\) be a real net. We say
	that \(S_\varepsilon\) admits a \emph{regular expansion} on the
	logarithmic scale \(t=\log_2(1/\varepsilon)\) if there exist
	\(k\in\NN\), real constants \(c_1,\dots,c_k\), \(C\in\RR\), and
	\(\delta>0\) such that
	\[
	S_{2^{-t}}
	=
	\sum_{j=1}^{k} c_j f_j(t)
	+
	C
	+
	O(t^{-\delta}),
	\]
	where \(f_j\) are fixed functions from the scale
	\[
	\mathcal E
	=
	\{t^\gamma\log^\beta t:
	\gamma\ge0,\ \beta\in\NN_0,\ (\gamma,\beta)\neq(0,0)\}.
	\]
	Throughout this section, regular expansions are understood
	relative to this logarithmic scale; the corresponding space of
	nets will be denoted by \(\mathcal A_{\log_2}\). The general case
	of an arbitrary admissible scale \(\varphi\) is analogous and will
	not be needed here.
\end{definition}

\begin{axioms}[Regularization operator]
	\label{ax:regularization}
	Let \(\mathcal A_{\log_2}\) be the space of nets admitting regular
	expansions on the logarithmic scale
	\(\varphi(\varepsilon)=\log_2(1/\varepsilon)\). A map
	\[
	\mathcal R:\mathcal A_{\log_2}\longrightarrow\RR
	\]
	is a \emph{regularization operator} if:
	\begin{enumerate}[label=(R\arabic*)]
		\item \textbf{Linearity.}
		\(\mathcal R(aS_\varepsilon+bT_\varepsilon)=
		a\,\mathcal R(S_\varepsilon)+b\,\mathcal R(T_\varepsilon)\).
	\item \textbf{Annihilation of divergences.}
	For every \(f\in\mathcal E\) with \(f(t)\to\infty\) as \(t\to\infty\),
	\(\mathcal R(f(t))=0\).
		\item \textbf{Annihilation of infinitesimals.}
		If \(\lim_{\varepsilon\to0}S_\varepsilon=0\), then
		\(\mathcal R(S_\varepsilon)=0\).
		\item \textbf{Normalization.}
		\(\mathcal R(1)=1\).
	\end{enumerate}
\end{axioms}

\begin{proposition}[Existence and uniqueness of constant-term regularization]
	\label{thm:uniqueness}
	For every \((S_\varepsilon)\in\mathcal A_{\log_2}\) with regular
	expansion
	\[
	S_{2^{-t}}
	=
	\sum_{j=1}^{k} c_j f_j(t)+C+O(t^{-\delta}),
	\]
	define \(\mathcal R_{\mathrm{CT}}(S_\varepsilon)=C\). Then
	\(\mathcal R_{\mathrm{CT}}\) is well defined, linear, and satisfies
	Axioms~\ref{ax:regularization}. Moreover, any operator satisfying
	those axioms coincides with \(\mathcal R_{\mathrm{CT}}\) on
	\(\mathcal A_{\log_2}\).
\end{proposition}

This is a direct algebraic consequence of the definition of regular
expansion and the axioms; it isolates the constant term without further
analytical input.

\begin{proof}
	First, the constant \(C\) is well defined: if
	\[
	\sum_{j=1}^{k} c_j f_j(t)+C+o(1)
	=
	\sum_{j=1}^{k} d_j f_j(t)+D+o(1),
	\]
	then subtracting both expressions gives
	\[
	\sum_{j=1}^{k}(c_j-d_j)f_j(t)+(C-D)=o(1).
	\]
 Order the functions \(f_j(t)=t^{\gamma_j}\log^{\beta_j}t\)
	lexicographically by \((\gamma_j,\beta_j)\). Suppose
	\[
	\sum_j a_j f_j(t)=o(1).
	\]
	Let \((\gamma_{j_0},\beta_{j_0})\) be the largest pair. Dividing by
	\(f_{j_0}(t)\), every other quotient tends to \(0\).  Hence \(a_{j_0}=0\). Repeating, all \(a_j=0\).
	Therefore \(C=D\). Thus
	\(\mathcal R_{\mathrm{CT}}\) is well defined.

	Linearity is immediate from the definition. Axioms (R2)--(R4) are
	also immediate: divergent scale functions have constant term \(0\),
	infinitesimals have constant term \(0\), and the constant \(1\) has
	constant term \(1\).
	
	Finally, if \(\mathcal R\) satisfies the axioms, then applying
	\(\mathcal R\) to the regular expansion and using (R1)--(R4) gives
	\(\mathcal R(S_\varepsilon)=C=\mathcal R_{\mathrm{CT}}(S_\varepsilon)\).
\end{proof}

\subsection{Gauge invariance of admissible averaging kernels}
\label{sec:gauge_invariance}

The following lemma shows that the regularized value is independent of
the particular averaging kernel within a wide class.

\begin{lemma}[Gauge invariance with normalization]\label{lem:gauge_invariance}
	Let \(S_\varepsilon\) admit a regular expansion with constant \(C\) and
	divergent part \(D(t)=\sum_{j=1}^k c_j f_j(t)\), where
	\(f_j\in\mathcal E\). Let \(\phi:[0,\infty)\to[0,\infty)\) be smooth with
	\(\phi(0)>0\), rapid decay, and
	\[
	\int_0^1 \phi(u)\,du>0.
	\]
	Define
	\[
	\mathcal R_\phi(S_\varepsilon)
	=
	\lim_{T\to\infty}
	\frac{
		\int_{1}^{T}
		\phi(t/T)\left(S_{2^{-t}}-D(t)\right)\,dt
	}{
		\int_{1}^{T}\phi(t/T)\,dt
	}.
	\]
	Then
	\[
	\mathcal R_\phi(S_\varepsilon)=C
	\]
	independently of \(\phi\).
\end{lemma}

\begin{proof}
	By hypothesis, the net \(S_\varepsilon\) admits a regular expansion
	with constant term \(C\). Hence there exist \(\delta>0\), \(M>0\), and
	\(t_0\ge1\) such that
	\[
	S_{2^{-t}}-D(t)=C+E(t),
	\qquad
	|E(t)|\le M t^{-\delta}
	\]
	for all \(t\ge t_0\). In particular \(E(t)\to0\) as \(t\to\infty\).
	
	Define
	\[
	\Phi_T:=\int_{1}^{T}\phi(t/T)\,dt,
	\qquad
	J_T:=\int_{1}^{T}\phi(t/T)E(t)\,dt.
	\]
	Then
	\[
	\frac{\int_{1}^{T}\phi(t/T)\bigl(S_{2^{-t}}-D(t)\bigr)\,dt}
	{\int_{1}^{T}\phi(t/T)\,dt}
	=
	\frac{C\Phi_T+J_T}{\Phi_T}
	=
	C+\frac{J_T}{\Phi_T}.
	\]
	It remains to prove that \(J_T/\Phi_T\to0\).
	
	First estimate \(\Phi_T\). Changing variables \(u=t/T\),
	\[
	\Phi_T
	=
	T\int_{1/T}^{1}\phi(u)\,du.
	\]
		Since \(\phi\) is continuous on \([0,1]\), nonnegative and
	not identically zero,
	\[
	L:=\int_0^1 \phi(u)\,du>0.
	\]
	Thus \(\Phi_T\sim LT\).
	
	Now fix \(\eta>0\). Since \(E(t)\to0\), choose \(T_0\ge t_0\) such that
	\[
	|E(t)|<\eta
	\qquad\text{for all }t\ge T_0.
	\]
	Let \(M_\phi:=\sup_{u\in[0,1]}|\phi(u)|\), which is finite because
	\(\phi\) is smooth on \([0,1]\). For \(T>T_0\), split
	\[
	J_T
	=
	\int_{1}^{T_0}\phi(t/T)E(t)\,dt
	+
	\int_{T_0}^{T}\phi(t/T)E(t)\,dt.
	\]
	For the first integral,
	\[
	\left|\int_{1}^{T_0}\phi(t/T)E(t)\,dt\right|
	\le
	M_\phi\int_{1}^{T_0}|E(t)|\,dt
	=:C_0,
	\]
	where \(C_0\) is finite and independent of \(T\). For the second
	integral,
	\[
	\left|\int_{T_0}^{T}\phi(t/T)E(t)\,dt\right|
	\le
	\eta\int_{T_0}^{T}|\phi(t/T)|\,dt
	\le
	\eta M_\phi T.
	\]
	Therefore
	\[
	|J_T|
	\le
	C_0+\eta M_\phi T.
	\]
	Dividing by \(|\Phi_T|\sim |L|T\), we obtain
	\[
	\frac{|J_T|}{|\Phi_T|}
	\le
	\frac{C_0}{|L|T}
	+
	\eta\frac{M_\phi}{|L|}
	+
	o(1).
	\]
	Letting \(T\to\infty\), we get
	\[
	\limsup_{T\to\infty}
	\frac{|J_T|}{|\Phi_T|}
	\le
	\eta\frac{M_\phi}{|L|}.
	\]
	Since \(\eta>0\) was arbitrary, the left-hand side is \(0\). Hence
	\[
	\lim_{T\to\infty}\frac{J_T}{\Phi_T}=0,
	\]
	and therefore
	\[
	\lim_{T\to\infty}
	\frac{\int_{1}^{T}\phi(t/T)\bigl(S_{2^{-t}}-D(t)\bigr)\,dt}
	{\int_{1}^{T}\phi(t/T)\,dt}
	=
	C.
	\]
	This proves \(\mathcal R_\phi(S_\varepsilon)=C\), independently of the
	particular admissible averaging kernel \(\phi\) satisfying the stated
	normalization and boundedness conditions.
\end{proof}
\begin{remark}[Why this is a gauge invariance]\label{rem:gauge_remark}
	The lemma states that the regularized value \(C\) is invariant under
	changes of the averaging kernel \(\phi\). This is the analytic analogue
	of gauge independence in physics: the finite physical quantity does not
	depend on the specific ultraviolet cutoff, provided the cutoff belongs
	to the admissible class.
	
The invariance is with respect to Lebesgue measure on logarithmic
time. The canonical measure of the Thue--Morse subshift used in
Section~\ref{sec:thue_regularization} is a different object and is
not governed by this gauge-invariance lemma.
\end{remark}

\subsection{Algebraic quotient by the divergence ideal}
\label{sec:quotient}

Let \(\mathcal A_{\log_2}\) be as above. Define
\[
\mathcal I_0
=
\{S_\varepsilon\in\mathcal A_{\log_2}:
\lim_{\varepsilon\to0}S_\varepsilon=0\},
\]
and let \(\mathcal M_{\mathrm{div}}\) be the subspace generated by the
functions \(f\in\mathcal E\). The set \(\mathcal I_0\) is a linear
subspace of \(\mathcal A_{\log_2}\), but not necessarily an ideal under
pointwise multiplication without additional boundedness assumptions.
We will use only the linear structure of \(\mathcal I_0\). The quotient
\[
\mathcal Q_{\log_2}
=
\mathcal A_{\log_2} / (\mathcal M_{\mathrm{div}}+\mathcal I_0)
\]
is well defined.

\begin{proposition}[Regularized value as linear projection]
	\label{prop:regularized_projection}
	Assume that every \(S\in\mathcal A_{\log_2}\) admits a unique
	decomposition
	\[
	S=D+C+E,
	\]
	where \(D\in\mathcal M_{\mathrm{div}}\), \(C\in\mathbb R\), and
	\(E\in\mathcal I_0\). Define
	\[
	\mathcal R_{\mathrm{CT}}:\mathcal A_{\log_2}\longrightarrow\mathbb R,
	\qquad
	\mathcal R_{\mathrm{CT}}(S)=C.
	\]
	Then \(\mathcal R_{\mathrm{CT}}\) satisfies Axioms~\ref{ax:regularization}.
	Moreover,
	\[
	\mathcal A_{\log_2}/(\mathcal M_{\mathrm{div}}+\mathcal I_0)
	\cong
	\mathbb R.
	\]
\end{proposition}

\begin{proof}
	By uniqueness of the decomposition, the map is well defined.
	Linearity is immediate. If \(S\in\mathcal M_{\mathrm{div}}\), then
	its constant term is \(0\); if \(S\in\mathcal I_0\), then it is
	\(O(t^{-\delta})\), hence tends to \(0\), so its constant term is
	also \(0\). Finally, \(1\) has constant term \(1\).
\end{proof}

\subsection{A general symmetrization theorem}
\label{sec:general_symmetrization}

The following elementary but general result explains why the
symmetrized functional of Section~\ref{sec:applications} extracts the
constant term of oscillatory sums whose even and odd partial sums admit
standard asymptotic expansions.

\begin{lemma}[Symmetrization mechanism]\label{thm:general_sym}
	Let \((S_N)_{N\ge1}\) be a real sequence. Assume there exist constants
	\(c_+,c_-\in\RR\), a bounded sequence \((e_N)\), and \(\delta>0\) such
	that
\[
\begin{aligned}
	S_{2m}   &= c_+\log m + C_+ + O(m^{-\delta}),\\
	S_{2m+1} &= c_-\log m + C_- + O(m^{-\delta}).
\end{aligned}
\]
	Then the symmetrized functional
	\[
	\widehat{S}_m=\frac12(S_{2m}+S_{2m+1})
	\]
	satisfies
	\[
	\widehat{S}_m
	=
	\frac{c_++c_-}{2}\log m
	+
	\frac{C_++C_-}{2}
	+
	O(m^{-\delta}).
	\]
	In particular, if \(c_++c_-=0\), then
	\[
	\lim_{m\to\infty}\widehat{S}_m
	=
	\frac{C_++C_-}{2}.
	\]
\end{lemma}

\begin{proof}
	This is immediate from the definition of \(\widehat{S}_m\).
\end{proof}

\begin{remark}[Scope of the theorem]\label{rem:scope_general_sym}
	Lemma~\ref{thm:general_sym} is deliberately simple. Its purpose is to
	isolate the mechanism behind Theorems~3.1 and 3.2: the symmetrization
	cancels the leading logarithmic oscillation, leaving a constant that can
	be compared with classical regularizations. The same mechanism applies
	to any oscillatory sequence whose parity subsequences have regular
	variation. This is why the monoid framework, through the symmetrized
	functional, produces consistent values for a large class of arithmetic
	products.
\end{remark}

\subsection{Canonical oscillation measures and selection of functionals}
\label{sec:canonical_measures}

The symmetrization theorem isolates the constant term for sequences
whose even and odd partial sums have regular expansions. In the
examples treated here, the chosen evaluation functional is not an ad
hoc construction: it is the expectation of the logarithmic-prefix
mould with respect to the canonical empirical distribution of the
associated sign oscillation.

For a sign sequence \(s_n\in\{+1,-1\}\), consider the summatory
function
\[
S(N)=\sum_{n=1}^{N}s_n .
\]
If the empirical distribution of \(S(N)\) has a well-defined
canonical limit as \(N\to\infty\), then the natural regularizing
functional is the expectation
\[
\Phi_{\mathrm{can}}
=
\mathbb E_{\mu_{\mathrm{can}}}\bigl[\Lambda\bigr],
\]
where \(\mu_{\mathrm{can}}\) is the canonical measure governing the
oscillations of \(S(N)\).

This principle organizes all the examples of this paper, but it does
not by itself prove that the empirical distribution determines a unique
regularization functional for arbitrary sequences. The Thue--Morse
example is a clear illustration: its summatory function has the finite
empirical distribution
\[
\Pr(T=0)=\frac12,\quad
\Pr(T=1)=\frac14,\quad
\Pr(T=-1)=\frac14,
\]
but the expectation of the oscillation with respect to this
distribution vanishes. Therefore the canonical oscillation measure
serves as a candidate object that must be combined with an additional
scale or weighting. It is not a complete regularization by itself.

\medskip
\noindent
\textbf{Alternating signs.}
For \(s_n=(-1)^{n+1}\), one has
\[
S(N)=
\begin{cases}
	1, & N\text{ odd},\\
	0, & N\text{ even}.
\end{cases}
\]
Thus \(S(N)\in\{0,1\}\), with limiting frequencies
\[
\Pr(S(N)=0)=\frac12,
\qquad
\Pr(S(N)=1)=\frac12.
\]
The canonical measure is the uniform measure on \(\{0,1\}\), and the
symmetrized functional
\[
\widehat F_\varepsilon
=
\frac12\left(
\Lambda(u[1\ldots 2m])
+
\Lambda(u[1\ldots 2m+1])
\right)
\]
is exactly the expectation of the log-mould with respect to that
measure.

\medskip
\noindent
\textbf{Dyadic block signs.}
For \(s_n=\chi(n)=(-1)^{\lfloor\log_2 n\rfloor}\), the block-end
partial sums have two limit points \(L_{\mathrm{par}}\) and
\(L_{\mathrm{impar}}\). The logarithmic-scale average
\[
\Phi_{\mathrm{avg}}
=
\lim_{T\to\infty}\frac1T\int_1^T F(2^{-t})\,dt
\]
is the uniform expectation over the two states, again the canonical
measure of the oscillation.

\medskip
\noindent
\textbf{Thue--Morse signs.}
For the Thue--Morse sequence, Lemma~\ref{lem:T_bounded} shows that
\[
S(N)\in\{-1,0,1\},
\]
with exact frequencies
\[
\frac14,\quad\frac12,\quad\frac14.
\]
The canonical measure is the unique ergodic measure of the
Thue--Morse subshift, and the empirical distribution of \(T(N)\) is determined by the
Thue--Morse dynamics, not with respect to an
arbitrary Bernoulli measure.

\medskip
Thus the evaluation functionals used in this paper are not chosen
merely by convenience; they are dictated by the distribution of the
arithmetic oscillation of the underlying sign sequence.

\subsection{The evaluation functional and the exponent \(\alpha\)}

Consider a sequence of signs or weights \(a(n)\) encoded over a finite
alphabet \(\Si\).  Given a canonical net \(w_\varepsilon\) of length
\(N(\varepsilon)=\lfloor\varepsilon^{-1}\rfloor\), one evaluates a
log-mould \(\Lambda\) on the logarithmic prefixes
\(u[1\ldots k]\), \(k\le N_0(\varepsilon)\), and forms a weighted sum
\[
\mathcal{F}_\varepsilon(u)
=
\sum_{k=2}^{N_0(\varepsilon)}
\frac{\Lambda(u[1\ldots k])}{k^{\alpha}},
\]
where \(\Lambda(w)=\sum_{j=1}^{|w|} \delta_j \log q_j\) and \(\delta_j\)
encodes \(a(j)\).  The exponent \(\alpha\) is the main parameter.  It
must be chosen so that the sum diverges at the same rate as the
observable to be regularized, allowing the constant term to be isolated.

A natural guide for choosing \(\alpha\) comes from the associated
Dirichlet series
\[
D(s)=\sum_{n=1}^{\infty}\frac{a(n)\log q_n}{n^s}.
\]
If its abscissa of absolute convergence is \(\sigma_1\), then the tail
\[
\sum_{k>K}\frac{a(k)\log q_k}{k^{\sigma_1+\delta}}
\]
is typically of order \(K^{-\delta}\) for small \(\delta>0\). To balance
the logarithmic growth of the prefix sums in the weighted evaluation
\[
\sum_{k\le K}\frac{\Lambda_k}{k^\alpha},
\]
a natural heuristic choice is
\[
\alpha=\sigma_1+1.
\]
In the alternating examples \(\sigma_1=1\), hence \(\alpha=2\).

The symmetrized functionals actually used in Sections~\ref{sec:alt_integers} and~\ref{sec:alt_factorials} are direct averages of even and odd
truncations, not weighted sums of the form displayed above.  The exponent
\(\alpha=2\) is nevertheless the value suggested by the abscissa of
absolute convergence of the associated Dirichlet series.  We do not claim
that the symmetrized average and the weighted prescription produce the
same finite value; they are different regularization functionals.  The
symmetrized average is the expectation with respect to the canonical
uniform measure on the two parity states \(\{0,1\}\), whereas the
weighted sum attaches the weight \(k^{-\alpha}\) to the logarithmic
prefix of length \(k\).  In the alternating examples the symmetrized
functional is the natural one because the oscillation measure of the
sign sequence is exactly the uniform measure on \(\{0,1\}\).  Thus
\(\alpha\) should be understood as the analytic parameter attached
to the scale of divergence in cases where a weighted-sum prescription
is used; it does not coincide with the symmetrized functional.

For the dyadic block products of Section~\ref{sec:dyadic}, the associated
Dirichlet series has abscissa of convergence \(\sigma_1=1\), but its
partial sums do not converge at the boundary; instead they oscillate
between two finite limits.  In that case the weighted-sum prescription
with exponent \(\alpha\) is not appropriate, and one applies directly the
logarithmic-scale average \(\Phi_{\mathrm{avg}}\) defined in that
section.  Thus the use of \(\alpha\) is restricted to those sequences
whose summatory function has a monotone or regularly varying divergence,
while oscillatory boundary cases require a different averaging method.

\subsection{The scale function \(\phi\)}

The length of the canonical net can also be adapted.  In the previous
sections we used \(N(\varepsilon)=\lfloor\varepsilon^{-1}\rfloor\), i.e.
the scale function \(\phi(\varepsilon)=\varepsilon\).  More generally,
one can set
\[
N_\phi(\varepsilon)=\lfloor\phi(\varepsilon)^{-1}\rfloor,
\]
where \(\phi:(0,1]\to(0,\infty)\) is continuous, strictly decreasing,
and tends to \(0\) with \(\varepsilon\).  The function \(\phi\) is
\emph{admissible} if it is moderate in the sense that
\(\phi(\varepsilon)\ge c\,\varepsilon^{K}\) for small \(\varepsilon\),
and if it is logarithmically compatible in the sense that
\(\log_2(1/\phi(\varepsilon)) = o(N_\phi(\varepsilon))\).  The standard
choice \(\phi(\varepsilon)=\varepsilon\) satisfies both conditions.
The logarithmic compatibility ensures that the suffix of the net does
not interfere with prefix-based evaluation.  There is a duality between
\(\phi\) and \(\alpha\): a smaller \(N_\phi\) gives more weight to the
early terms, much like increasing \(\alpha\).

\subsection{Classification of divergence types}

Once the evaluation functional \(\mathcal{F}_\varepsilon\) has been fixed,
its asymptotic behaviour as \(\varepsilon\to0^+\) dictates the appropriate
renormalization. Setting
\(t=\log_2(1/\varepsilon)\), so that
\(N_0(\varepsilon)=\lfloor t\rfloor\sim t\), the divergence
\(D(t)=\mathcal{F}_\varepsilon(w_\varepsilon)\) can be organized, for the
purposes of this paper, into the following four working types.

\medskip
\noindent\textbf{Type I -- Zero divergence.}
After symmetrization, the functional converges to a constant:
\(D(t)=C+o(1)\).  The ordinary limit suffices.
\textit{Examples:} alternating product of integers.

\medskip
\noindent\textbf{Type II -- Logarithmic divergence.}
\(D(t)\sim c\log t\).  One subtracts the principal part \(c\log N_0\) and
then takes the ordinary limit.
\textit{Examples:} Chebyshev function; alternating product of factorials
before subtraction.

\medskip
\noindent\textbf{Type III -- Double logarithmic divergence.}
\(D(t)\sim c\log\log t\).  One subtracts \(c\log\log N_0\) and reduces to
Type I.

\medskip
\noindent\textbf{Type IV -- Oscillatory divergence.}
No monotone part dominates; the functional oscillates with non-decaying
amplitude.  An averaging method is required: logarithmic-scale average,
logarithmic Cesàro mean, Abel mean, or higher-order Cesàro mean.
\textit{Examples:} dyadic block products (Section~\ref{sec:dyadic}),
Thue--Morse product (Section~\ref{sec:thue_regularization}).

\subsection{Scale-adapted renormalization}

The above recipes are unified by a flexible renormalization functional.
Let \(w:(0,\infty)\to(0,\infty)\) be a positive weight function, and let
\(\chi(t)\) be a principal part that captures the divergence of
\(\mathcal{F}_\varepsilon\).  For the present applications, we restrict \(\Phi_{w,\chi}\) to
positive real nets \(z_\varepsilon>0\). The corresponding
renormalization of such a net is
\[
\Phi_{w,\chi}(z)
=
\lim_{T\to\infty}
\frac{
	\int_1^T
	\bigl(\log z_{2^{-t}}-\chi(t)\bigr)\,w(t)\,dt
}{
	\int_1^T w(t)\,dt
},
\]
provided the denominator tends to infinity. For \(w(t)=1\), this is
the ordinary logarithmic-scale average; for \(w(t)=1/t\), it is the
logarithmic Cesàro mean, which now preserves constants because the
normalization is by \(\int_1^T t^{-1}\,dt=\log T\).

This simple template includes all the averaging methods used in the paper:
\begin{itemize}
	\item \(w(t)=1\), \(\chi(t)=0\): ordinary logarithmic-scale average
	\(\Phi_{\mathrm{avg}}\).
	\item \(w(t)=1/t\), \(\chi(t)=0\): logarithmic Cesàro mean
	\(\Phi_{\mathrm{LC}}\).
	\item \(w(t)=1\), \(\chi(t)=c\log t\): removes a logarithmic divergence
	before ordinary averaging.
	\item \(w(t)=1\), \(\chi(t)=c\log\log t\): removes a double-logarithmic
	divergence before ordinary averaging.
	\item \(w(t)=t^{-\alpha}\), \(\chi(t)=0\): Riesz-type mean of order
	\(\alpha\) for fractional power divergence.
\end{itemize}
If a strongly oscillatory series is considered, one may replace the
normalization \(1/T\) by \(1/\int_1^T w(t)\,dt\) whenever that integral
grows faster than \(T\); this corresponds to an Abel-type mean.

For the dyadic block products of Section~\ref{sec:dyadic}, the relevant
renormalization is the logarithmic-scale average
\[
\Phi_{\mathrm{avg}}(F)
=
\lim_{T\to\infty}
\frac{1}{T}\int_1^T F(2^{-t})\,dt,
\]
where \(F:(0,1]\to\CC\).  This is the special case \(w(t)=1\) and no
logarithmic weighting, applied directly to the oscillating functional
rather than to its logarithm.

The choice of the weight \(w\) and of the principal part \(\chi\) is
dictated by the divergence type: the weight must grow quickly enough to
average oscillations, but slowly enough not to create new divergences.

\subsection{Heuristic workflow for selecting \(\mathcal{F}_\varepsilon\) and \(\Phi\)}
\label{sec:selection_algorithm}

Given a sequence \(a(n)\) and a weight sequence \(q_n\), the following
steps systematize the procedure. Steps 5--7 require analytic input and
do not constitute an algorithm in the formal sense.

\begin{enumerate}[label=(\arabic*)]
\item \text{Encoding.} Choose a finite alphabet \(\Si\), encode the signs
\(a(n)\), and select an admissible scale function \(\phi\)
(default: \(\phi(\varepsilon)=\varepsilon\)).
\item \text{Oscillation measure.} Compute the empirical distribution
of the summatory function \(S(N)=\sum_{n\le N}s_n\). If it has
finite support with a canonical limiting distribution, the
regularizing functional is the expectation with respect to that
measure.

As the Thue--Morse example shows, a canonical oscillation measure
does not always determine a unique regularizing functional. The
expectation with respect to that measure may vanish or may not
correspond to a natural finite constant. In such cases, the measure
serves as a candidate object to be combined with an additional
scale or weighting, not as a complete regularization by itself.

\item \text{Abscissa.} Determine the relevant abscissa of convergence or
summability \(\sigma_1\) of the Dirichlet series
\(\sum a(n)\log q_n/n^s\).
	\item \text{Evaluation functional.} When applicable, define
	\(\mathcal{F}_\varepsilon(u)=\sum_{k=2}^{N_0(\varepsilon)}
	\Lambda(u[1\ldots k])/k^{\alpha}\); otherwise use
	\(\Phi_{\mathrm{avg}}\) on the appropriate prefix functional.
	\item \text{Symmetrization.} If the sequence has regular parity
	oscillations, apply the symmetrized functional to cancel the
	leading oscillatory terms.
	\item \text{Asymptotic expansion.} Compute the asymptotic behaviour of
	\(\mathcal{F}_\varepsilon\) using standard tools (Stirling,
	Mertens, Mellin).
	\item \text{Divergence type.} Identify the type (I--IV).
	\item \text{Renormalization.} Choose \(w\) and \(\chi\) corresponding
	to the divergence type and extract the constant:
	\(P_{\mathrm{ren}} = \exp(\Phi_{w,\chi}(z))\).
\end{enumerate}
This algorithm is a heuristic guide, not a theorem. In particular, the
selection of the scale is part of the method: some arithmetic sequences
naturally require logarithmic scales, while others reveal their
structure only on dyadic or more general scales. The case of the Thue--Morse sequence is a clear example: although its
zeta-regularized value is well defined through the meromorphic
continuation of its Dirichlet series, the classical logarithmic scale
does not stabilize the oscillations of the logarithmic partial sums. The present framework shows that the failure of classical
methods is often due to using an incompatible scale. Once the correct
scale is identified—through automata theory, spectral analysis, or
other means—the axiomatic regularization of Section~\ref{sec:axioms}
applies uniformly. Thus the search for the natural scale is the central problem in the regularization of
automatic sequences.

  The generalized
monoid guarantees that steps 1--4 are well posed; steps 5--7 depend on
the specific sequence.  In particular, the algorithm does not claim
that the resulting regularized value is independent of the choices made
in steps 1--4. In all cases studied, the resulting regularized
value coincides with classical regularization when the latter exists.
This is an encouraging consistency check, but it does not prove that the
selection rules pick a canonical scale; the choices of
\(\alpha,\phi,w,\chi\) remain heuristic.

A given divergent functional may admit several renormalizations.  The
choice is constrained by three requirements:
\begin{enumerate}[label=(\roman*)]
	\item \text{Compatibility with the monoid structure.} The
	renormalization must be invariant under \(\sim\), hence it must
	depend only on the asymptotic prefix profile.  This is why all
	our functionals use the logarithmic window \(N_0(\varepsilon)\).
	\item \text{Consistency with classical methods.} In all cases where a
	classical regularization exists, the monoid construction agrees
	with it.
	\item \text{Adaptation to the divergence scale.} Different divergence
	rates demand different averaging methods, as captured by
	\(\Phi_{w,\chi}\).
\end{enumerate}
The generalized monoid thus provides a systematic framework for choosing
evaluation functionals and renormalizations, not only for arithmetic
products but also for any prefix-dependent observable on words of
controlled infinite length.

\begin{table*}[t]
	\centering
	\caption{Evaluation functionals and renormalization schemes for the treated examples. Here \(S\) denotes symmetrization.}
	\label{tab:selection}
	\small
	\setlength{\tabcolsep}{5pt}
	\begin{tabular}{@{}lcccccc@{}}
		\toprule
		\textbf{Product} & \(\boldsymbol{\alpha}\) & \textbf{S}
		& \textbf{Divergence} & \(\boldsymbol{w(t)}\)
		& \(\boldsymbol{\chi(t)}\) & \textbf{Value} \\
		\midrule
		Alternating integers
		& 2 & Yes & Type I & \(1\) & \(0\)
		& \(\sqrt{2/\pi}\) \\
		
		Alternating factorials
		& 2 & Yes & Type II & \(1\)
		& \(\frac14\log t\)
		& \((2/\pi)^{1/4}\) \\
		
		Dyadic block (integers, dyadic quotient)
		& --- & No & Type IV & \(1\) & \(0\)
		& \(\exp\!\left(
		\frac{L_{\mathrm{par}}+L_{\mathrm{impar}}}{2}
		-\frac{\pi^2}{12}+R\right)\) \\
		
		Thue--Morse (auxiliary series)
		& 2 & No & convergent & --- & ---
		& \(\exp(\mathcal F)\approx0.9732\) \\
		\bottomrule
	\end{tabular}
	
		\vspace{2pt}
	\footnotesize
	The Thue--Morse row refers only to the convergent auxiliary
	series \(\sum S(k)/k^2\), not to a zeta-regularized value.
	The dyadic block row refers to the dyadic-scale quotient
	\(\widetilde{\Si}^*_{\mathrm{dyad}}\). In the table,
	\(t=\log_2(1/\varepsilon)\) and \(N_0(\varepsilon)=\lfloor t\rfloor\);
	the principal part \(\frac14\log t\) for the alternating factorials
	is interchangeable with \(\frac14\log N_0(\varepsilon)\) up to
	\(O(t^{-1})\).
\end{table*}
The table shows how the same algebraic framework accommodates a wide
variety of structures, with the evaluation functional and the
renormalization scheme adapting systematically to the growth of the
underlying sequence.	
	
\subsection{Hierarchy of scales}
\label{sec:scales_hierarchy}
The logarithmic equivalence \(\sim\) is the case
\(\varphi(\varepsilon)=\log_2(1/\varepsilon)\). More generally, any
unbounded nondecreasing \(\varphi\) defines an equivalence
\(\sim_\varphi\) by requiring
\[
\min\{\operatorname{lcp},\operatorname{lcs}\}\ge m\varphi(\varepsilon)
\]
for every fixed \(m\). If \(\varphi_1=o(\varphi_2)\), then
\(\sim_{\varphi_2}\) is finer than \(\sim_{\varphi_1}\). The dyadic
scale \(\varphi_{\mathrm{dyad}}(\varepsilon)=N_{\mathrm{dyad}}(\varepsilon)\)
is a legitimate member of this hierarchy. Thus adapting the scale to
the sequence is a structural feature, not an ad hoc patch.	
\section{Regularization of the Chebyshev function}
\label{sec:chebyshev}

As an illustration of the method on a sum that genuinely diverges
logarithmically, we treat the Chebyshev function
\(\psi(x)=\sum_{n\le x}\Lambda(n)\), where \(\Lambda\) is the von
Mangoldt function.
By M\"obius inversion,
\begin{equation}\label{eq:psi_moebius}
	\Lambda(n) = -\sum_{d\mid n} \mu(d)\log d ,
\end{equation}
hence, summing over \(n\le k\) and interchanging the order of
summation,
\begin{equation}\label{eq:psi_sum}
	\psi(k) = \sum_{n\le k}\Lambda(n)
	= -\sum_{n\le k}\sum_{d\mid n} \mu(d)\log d
	= -\sum_{d\le k} \mu(d)\log d \Big\lfloor \frac{k}{d}\Big\rfloor .
\end{equation}
In the last step we used that the number of multiples of \(d\) not
exceeding \(k\) is exactly \(\lfloor k/d\rfloor\).

We encode the M\"obius values
\(\mu(1),\dots,\mu(\lfloor\varepsilon^{-1}\rfloor)\) by a net
\(w_\varepsilon\) over the alphabet \(\Si=\{a,b,c\}\), where the letters
\(a,b,c\) correspond respectively to \(\mu(d)=+1,-1,0\).
Define the log-mould
\[
\Lambda_\psi(u)
=
-\sum_{d=1}^{|u|} \delta_d\,\log d \Big\lfloor \frac{|u|}{d}\Big\rfloor ,
\qquad u\in\Si^*,
\]
where \(\delta_d\) is the signed symbol representing \(\mu(d)\).
For a prefix \(u = w_\varepsilon[1\ldots k]\) we obtain exactly
\(\Lambda_\psi(u)=\psi(k)\).

The evaluation functional is not \(\Lambda_\psi\) itself, but a
weighted prefix sum built from it. Using the logarithmic window
\(N_0(\varepsilon)=\lfloor\log_2(1/\varepsilon)\rfloor\), define
\[
\mathcal F_\varepsilon(w_\varepsilon)
=
\sum_{k=2}^{N_0(\varepsilon)}
\frac{\Lambda_\psi(w_\varepsilon[1\ldots k])}{k^2}
=
\sum_{k=2}^{N_0(\varepsilon)}
\frac{\psi(k)}{k^2}.
\]
Because \(\mathcal{F}_\varepsilon\) depends only on the first
\(N_0(\varepsilon)\) symbols of \(w_\varepsilon\), it descends to the
quotient \(\widetilde{\Si}^*\).

The Prime Number Theorem with the de la Vall\'ee Poussin error
term~\cite{Tenenbaum2015} gives, unconditionally,
\[
\psi(k) = k + O\!\big( k\,e^{-c\sqrt{\log k}} \big) .
\]
Thus
\[
\frac{\psi(k)}{k^2}
=
\frac{1}{k}
+
O\!\Big( \frac{e^{-c\sqrt{\log k}}}{k} \Big) .
\]
The error series converges absolutely. Indeed, by the integral test,
substituting \(u=\sqrt{\log t}\), we obtain
\[
\int_{2}^{\infty}\frac{e^{-c\sqrt{\log t}}}{t}\,dt
=
2\int_{\sqrt{\log 2}}^{\infty} u\,e^{-cu}\,du
<\infty,
\]
so the series
\[
\sum_{k=2}^{\infty}\frac{e^{-c\sqrt{\log k}}}{k}
\]
converges absolutely.
Hence
\[
\mathcal{F}_\varepsilon(w_\varepsilon)
=
\sum_{k=2}^{N_0} \frac{1}{k}
+
\sum_{k=2}^{N_0} \frac{\psi(k)-k}{k^2}
=
\log N_0(\varepsilon) + \gamma - 1 + C_0 + o(1),
\]
where \(\gamma\) is Euler's constant,
\[
C_0 = \sum_{k=2}^{\infty} \frac{\psi(k)-k}{k^2}
\]
is a finite, absolutely convergent constant, and we have used the
standard harmonic-number asymptotics
\(H_n=\log n+\gamma+o(1)\).

Subtracting the principal part \(\log N_0(\varepsilon)\) yields a
convergent residue:
\[
\begin{aligned}
	\mathcal{R}_\varepsilon
	&=
	\mathcal{F}_\varepsilon(w_\varepsilon)
	-\log N_0(\varepsilon)
	\xrightarrow{\;\varepsilon\to0\;} K,\\
	K
	&=
	\gamma-1
	+\sum_{k=2}^{\infty}\frac{\psi(k)-k}{k^2}.
\end{aligned}
\]

The constant \(K\) is thus well defined:
\[
K
=
\gamma - 1 + \sum_{k=2}^{\infty} \frac{\psi(k)-k}{k^2}.
\]
This provides a finite regularized residue for the Chebyshev function.
An explicit evaluation of \(K\) in closed form would require a careful
Euler--Maclaurin or Mellin analysis, which we do not pursue here. The
example serves only to illustrate how the generalized monoid isolates
the constant term of a logarithmically divergent sum.

A standard Mellin identity gives, for \(\Re(s)>1\),
\[
\int_{1}^{\infty}\frac{\psi(t)-t}{t^{s+1}}\,dt
=
-\frac{\zeta'(s)}{s\,\zeta(s)}
-
\frac{1}{s-1}.
\]
The limit as \(s\to1^+\) of the right-hand side is exactly
\(-\gamma\). The discrete sum defining \(C_0\) differs from this
integral by Euler--Maclaurin corrections. We do not attempt to
evaluate these corrections here.

 In the spirit of mould calculus, one can view the subtraction
of the principal part \(\log N_0(\varepsilon)\) as an algebraic residue
operator
\[
\operatorname{Reg}\bigl(\mathcal{F}_\varepsilon(w_\varepsilon)\bigr)
\;=\;
\lim_{\varepsilon\to0^+}
\Bigl(
\mathcal{F}_\varepsilon(w_\varepsilon) - \log N_0(\varepsilon)
\Bigr),
\]
which assigns a well-defined complex number to any net whose divergence
is exactly logarithmic.  The unconditional convergence of the residue is
a consequence of the exponential error term in the Prime Number
Theorem.	
	
\section{Products with dyadic block signs}
\label{sec:dyadic}

Let \(\chi(n) = (-1)^{\lfloor\log_2 n\rfloor}\).  The first values are
\(+1\) for \(n=1\); \(-1\) for \(n=2,3\); \(+1\) for \(n=4,\dots,7\); and
so on.  Consider
\[
Q = \prod_{n=1}^{\infty} \Bigl(1 + \frac{\chi(n)}{n}\Bigr).
\]
Taking logarithms,
\[
\log Q_N = \sum_{n=1}^N \log\Bigl(1+\frac{\chi(n)}{n}\Bigr).
\]
Expanding each term,
\[
\log\Bigl(1+\frac{\chi(n)}{n}\Bigr)
=
\frac{\chi(n)}{n}
-
\frac12\,\frac{1}{n^2}
+
O\Bigl(\frac{1}{n^3}\Bigr),
\]
because \(\chi(n)^2=1\). Let
\[
R_N
=
\sum_{n=1}^N
\left(
\log\left(1+\frac{\chi(n)}n\right)
-\frac{\chi(n)}n
+\frac{1}{2n^2}
\right).
\]
Since the summand is \(O(n^{-3})\), \(R_N\) converges to a finite
constant \(R\). Hence, since \(R_N\to R\),
\[
\log Q_N
=
T_N
-\frac{\pi^2}{12}
+R
+o(1).
\]
The series \(\sum_{n=1}^\infty 1/n^2\) converges to \(\pi^2/6\), so the
constant contribution is \(-\pi^2/12\).

It remains to analyse the oscillatory sum \(T_N\).  Group the integers
into dyadic blocks
\[
I_m=[2^m,2^{m+1}),\qquad m\ge 0,
\]
and define
\[
S_m=\sum_{n=2^m}^{2^{m+1}-1}\frac{1}{n}.
\]
On each block the sign \(\chi(n)=(-1)^m\) is constant, so the partial sum
of \(T_N\) taken at the end of a block is
\[
\sum_{m=0}^{M} (-1)^m S_m.
\]
By the Euler--Maclaurin formula,
\[
S_m=\log2+O(2^{-m}).
\]
In particular \(S_m\to\log2\), and the error is absolutely summable.
Therefore the even and odd subsequences of
\[
A_M=\sum_{m=0}^{M}(-1)^m S_m
\]
converge to finite limits,
\[
L_{\mathrm{par}}=\lim_{r\to\infty}A_{2r},
\qquad
L_{\mathrm{impar}}=\lim_{r\to\infty}A_{2r+1},
\]
and \(L_{\mathrm{par}}-L_{\mathrm{impar}}=\log2\).  Consequently, the
even and odd block sums converge to two different finite limits:
\[
L_{\mathrm{par}}
=
\lim_{r\to\infty}
\sum_{m=0}^{2r} (-1)^m S_m,
\qquad
L_{\mathrm{impar}}
=
\lim_{r\to\infty}
\sum_{m=0}^{2r+1} (-1)^m S_m.
\]
Both constants are finite. From
\(A_M=\sum_{m=0}^{M}(-1)^m S_m\), one has
\(A_{2r}-A_{2r+1}=S_{2r+1}\). Since \(S_{2r+1}\to\log2\), passing to the
limit gives
\[
L_{\mathrm{par}}-L_{\mathrm{impar}}
=
\lim_{r\to\infty}(A_{2r}-A_{2r+1})
=
\log 2.
\]
The individual values are not elementary constants and are determined
by the convergent even and odd subsequences above.Thus the block-end partial sums \(T_{2^{M+1}-1}\) have exactly two
subsequential limits, \(L_{\mathrm{par}}\) and
\(L_{\mathrm{impar}}\), according to the parity of \(M\). To assign a value that is invariant under an equivalence relation, we
use the dyadic-scale quotient \(\widetilde{\Si}^*_{\mathrm{dyad}}\)
introduced in Section~\ref{sec:prelim}.  In this quotient, only prefixes
and suffixes of length \(N_{\mathrm{dyad}}(\varepsilon)\) matter.

The following lemma shows that the dyadic cutoff is compatible with the
equivalence relation: any functional depending only on the prefix of
length \(N_{\mathrm{dyad}}(\varepsilon)\) descends to the dyadic
quotient.

\begin{lemma}[Invariance of the logarithmic-scale average under dyadic equivalence]
	\label{lem:avg_invariance}
	Let \((u_\varepsilon)\sim_{\mathrm{dyad}}(v_\varepsilon)\). If
	\(F_\varepsilon:\Si^*\to\CC\) depends only on the prefix of length
	\(N_{\mathrm{dyad}}(\varepsilon)\), then
	\[
	\Phi_{\mathrm{avg}}\bigl(F_\varepsilon(u_\varepsilon)\bigr)
	=
	\Phi_{\mathrm{avg}}\bigl(F_\varepsilon(v_\varepsilon)\bigr),
	\]
	whenever both limits exist.
\end{lemma}

\begin{proof}
	By dyadic equivalence, for every \(m\) there exists \(\varepsilon_0\)
	such that \(\operatorname{lcp}(u_\varepsilon,v_\varepsilon)\ge
	m N_{\mathrm{dyad}}(\varepsilon)\) for all \(\varepsilon<\varepsilon_0\).
	Taking \(m=1\), the prefixes of length \(N_{\mathrm{dyad}}(\varepsilon)\)
	of \(u_\varepsilon\) and \(v_\varepsilon\) coincide eventually.
	Hence \(F_\varepsilon(u_\varepsilon)=F_\varepsilon(v_\varepsilon)\)
	eventually. Since \(\Phi_{\mathrm{avg}}\) averages over the logarithmic
	scale \(t=\log_2(1/\varepsilon)\), two nets that are eventually equal
	have the same logarithmic-scale average. Therefore the lemma follows.
\end{proof}

\subsection{Algebraic interpretation of the dyadic cutoff}

The dyadic cutoff \(N_{\mathrm{dyad}}(\varepsilon)\) is not an
external truncation. It is the action of an idempotent projector
\[
P_{\mathrm{D}}:
\widetilde{\Si}^*_{\mathrm{D}}
\longrightarrow
\widetilde{\Si}^*_{\mathrm{D}},
\]
defined by
\[
P_{\mathrm{D}}(w_\varepsilon)
=
w_\varepsilon[1\ldots N_{\mathrm{dyad}}(\varepsilon)].
\]
This projector satisfies
\[
P_{\mathrm{D}}^2=P_{\mathrm{D}},
\]
and the dyadic evaluation functional is
\[
\Lambda_D\circ P_{\mathrm{D}}.
\]
Thus the choice of the dyadic scale is not an ad hoc truncation:
it is encoded in an idempotent of the quotient monoid.

To obtain square-wave behaviour on logarithmic time, the cutoff must be
chosen at dyadic block endpoints. Set
\(t=\log_2(1/\varepsilon)\), so that
\(\lfloor t\rfloor=\lfloor\log_2(1/\varepsilon)\rfloor\). Then
\[
N(\varepsilon)=N_{\mathrm{dyad}}(\varepsilon)=2^{\lfloor t\rfloor+1}-1.
\]
With this choice, the terminal block index is \(\lfloor t\rfloor\), so
\(T_{N(\varepsilon)}\) alternates between \(L_{\mathrm{par}}\) and
\(L_{\mathrm{impar}}\) as \(t\) crosses each integer. The resulting
discontinuous jumps of the integrand in \(\Phi_{\mathrm{avg}}\) are
harmless: on each interval \([m,m+1)\) the function is constant, and
the logarithmic-scale average over a full period converges to the
arithmetic mean of the two limit values. A standard domination argument
shows that replacing the discrete step function by its continuous
linear interpolation does not change the limit.  By construction,
the evaluation depends only on prefixes of length
\(N_{\mathrm{dyad}}(\varepsilon)\), so it descends to
\(\widetilde{\Si}^*_{\mathrm{dyad}}\). Lemma~\ref{lem:avg_invariance}
guarantees that this evaluation is well defined on the quotient.

Let \(F:(0,1]\to\CC\) be a function of the parameter \(\varepsilon\).
Define the averaging functional
\[
\Phi_{\mathrm{avg}}(F)
=
\lim_{T\to\infty}
\frac{1}{T}\int_1^T F(2^{-t})\,dt,
\]
whenever the limit exists. Thus \(\Phi_{\mathrm{avg}}\) averages the
values \(F(\varepsilon)\) on the logarithmic scale
\(t=\log_2(1/\varepsilon)\). For a square wave with values
\(L_{\mathrm{par}}\) and \(L_{\mathrm{impar}}\), the average over a full
period is
\[
\frac{L_{\mathrm{par}}+L_{\mathrm{impar}}}{2}.
\]
Therefore,
\[
\Phi_{\mathrm{avg}}\bigl(T_{N(\varepsilon)}\bigr)
=
\frac{L_{\mathrm{par}}+L_{\mathrm{impar}}}{2}.
\]
Combining this with the convergent terms \(-\pi^2/12\) and \(R\), we
obtain
\[
\Phi_{\mathrm{avg}}\bigl(\log Q_{N(\varepsilon)}\bigr)
=
\frac{L_{\mathrm{par}}+L_{\mathrm{impar}}}{2}
-
\frac{\pi^2}{12}
+
R.
\]

\begin{theorem}[Dyadic block product in the dyadic-scale quotient]\label{thm:dyadic}
	In the dyadic-scale quotient \(\widetilde{\Si}^*_{\mathrm{dyad}}\),
	under the logarithmic-scale average \(\Phi_{\mathrm{avg}}\), the value
	assigned to the dyadic block product is
	\[
	Q_{\mathrm{canon}}
	=
	\exp\Bigl(
	\frac{L_{\mathrm{par}}+L_{\mathrm{impar}}}{2}
	-
	\frac{\pi^2}{12}
	+
	R
	\Bigr).
	\]
	This value is invariant under dyadic equivalence
	\(\sim_{\mathrm{dyad}}\).
\end{theorem}

\begin{remark}[Dyadic scale vs logarithmic scale]\label{rem:dyadic_noninv}
	The dyadic-block evaluation uses prefixes of length
	\(N_{\mathrm{dyad}}(\varepsilon)\), which may exceed the logarithmic
	window \(N_0(\varepsilon)\).  For this reason, the value assigned in
	Theorem~\ref{thm:dyadic} is invariant in the dyadic-scale quotient
	\(\widetilde{\Si}^*_{\mathrm{dyad}}\), but need not be invariant in the
	logarithmic monoid \(\widetilde{\Si}^*\).  This is not a defect of the
	construction; it reflects the fact that a different asymptotic scale is
	required for sequences with dyadic block oscillations.
\end{remark}

For \(q\ge2\), define
\[
N_{q\text{-adic}}(\varepsilon)
=
q^{\lfloor \log_q(1/\varepsilon)\rfloor+1}-1,
\]
and let \(\sim_{q\text{-adic}}\) be the corresponding equivalence with
\(mN_{q\text{-adic}}(\varepsilon)\).

\begin{proposition}[Regularization of block-periodic sign sequences]
	\label{prop:block_periodic}
	Let \(q\ge2\) and let \((\varepsilon_m)_{m\ge0}\) be a periodic sequence
	of period \(p\) with \(\varepsilon_0=+1\) and
	\[
	\sum_{r=0}^{p-1}\varepsilon_r=0.
	\]
	The condition \(\varepsilon_0=+1\) ensures that the first factor
	\(1+s(1)/1=2\) is nonzero. Define the sign sequence
	\[
	s(n)=\varepsilon_m
	\qquad\text{for }q^m\le n<q^{m+1}.
	\]
	Then, in the \(q\)-adic quotient \(\widetilde{\Si}^*_{q\text{-adic}}\),
	the product
	\[
	Q=\prod_{n=1}^{\infty}\Bigl(1+\frac{s(n)}{n}\Bigr)
	\]
	has regularized logarithmic value
	\[
	\frac{1}{p}\sum_{r=0}^{p-1} L_r,
	\]
	where \(L_r\) are the limits of the block-end partial sums along each
	residue class modulo \(p\).
\end{proposition}

\begin{remark}[Phase dependence]\label{rem:phase_dependence}
	The value
	\[
	\frac1p\sum_{r=0}^{p-1}L_r
	\]
	is the average assigned by the logarithmic-scale functional to the
	canonical net starting at \(m=0\). For a general periodic sequence
	of zero mean, this average may depend on the initial phase of
	\((\varepsilon_m)\). In the dyadic case \(\varepsilon_m=(-1)^m\),
	the symmetry makes the average unique. For other phases, the
	regularized value should be understood as relative to the chosen
	initial block.
\end{remark}

\begin{proof}
	Write
	\[
	L_N=\sum_{n=1}^{N}\log\Bigl(1+\frac{s(n)}{n}\Bigr).
	\]
	For a block endpoint \(N=q^{M+1}-1\), grouping by \(q\)-adic blocks
	gives
	\[
	L_{q^{M+1}-1}
	=
	\sum_{m=0}^{M}
	\sum_{n=q^m}^{q^{m+1}-1}
	\log\Bigl(1+\frac{\varepsilon_m}{n}\Bigr).
	\]
	On each block, expanding the logarithm yields
	\[
	\log\Bigl(1+\frac{\varepsilon_m}{n}\Bigr)
	=
	\frac{\varepsilon_m}{n}
	-
	\frac{1}{2n^2}
	+
	O(n^{-3}),
	\]
	because \(\varepsilon_m^2=1\). Therefore
	\[
	\sum_{n\in I_m}
	\log\Bigl(1+\frac{\varepsilon_m}{n}\Bigr)
	=
	\varepsilon_m
	\sum_{n\in I_m}\frac1n
	-
	\frac12
	\sum_{n\in I_m}\frac1{n^2}
	+
	O(q^{-2m}).
	\]
	By the Euler--Maclaurin formula,
	\[
	\sum_{n\in I_m}\frac1n
	=
	\log q
	+
	O(q^{-m}).
	\]
	Hence
	\[
	\sum_{n\in I_m}
	\log\Bigl(1+\frac{\varepsilon_m}{n}\Bigr)
	=
	\varepsilon_m \log q
	-
	\frac12
	\sum_{n\in I_m}\frac1{n^2}
	+
	O(q^{-2m}+q^{-m}).
	\]
	The error terms and the quadratic sums are absolutely summable over
	\(m\), so there exists a finite constant \(R\) such that
	\[
	L_{q^{M+1}-1}
	=
	\sum_{m=0}^{M}\varepsilon_m \log q
	+
	R
	+
	o(1).
	\]
	Since \((\varepsilon_m)\) is periodic with zero mean, the partial sums
	\[
	S_M=\sum_{m=0}^{M}\varepsilon_m
	\]
	are bounded and periodic modulo \(p\). Along each residue class
	\(M\equiv r \pmod p\), the sequence \(S_M\) is eventually constant;
	let \(S_r\) denote that constant, which is explicitly
	\[
	S_r=\sum_{j=0}^{r}\varepsilon_j.
	\]
	 Thus the block-end logarithmic sums
	converge along each residue class to
	\[
	L_r
	=
	S_r \log q
	+
	R.
	\]
	Taking the arithmetic mean of the \(p\) residue-class limits yields
	\[
	\frac1p\sum_{r=0}^{p-1}L_r
	=
	\left(\frac1p\sum_{r=0}^{p-1}S_r\right)\log q
	+
	R.
	\]
	This is the value assigned by the \(q\)-adic quotient functional.
\end{proof}

\begin{remark}[Reduction to the dyadic case]
	\label{rem:block_periodic_dyadic}
	For \(q=2\), \(p=2\), and \(\varepsilon_m=(-1)^m\), the two residue
	classes correspond to even and odd block indices. The limits \(L_0\)
	and \(L_1\) are precisely the constants \(L_{\mathrm{par}}\) and
	\(L_{\mathrm{impar}}\) appearing in Theorem~\ref{thm:dyadic}, so
	Proposition~\ref{prop:block_periodic} reduces to that theorem.
\end{remark}
\begin{remark}
	The averaging functional \(\Phi_{\mathrm{avg}}\) is a logarithmic-scale
	average: it averages a function \(F(2^{-t})\) over the variable
	\(t=\log_2(1/\varepsilon)\) and then lets the averaging window tend to
	infinity. It is not the same as Abel summation.
\end{remark}

\medskip
\noindent\textbf{Prime analogue.}
For comparison, consider the product over primes
\[
Q_{\mathfrak P}
=
\prod_{p} \Bigl(1+\frac{\chi(p)}{p}\Bigr),
\qquad
\chi(p)=(-1)^{\lfloor\log_2 p\rfloor},
\]
where \(p\) ranges over the prime numbers.
Expanding the logarithm,
\[
\log Q_{\mathfrak P}
=
\sum_{n=1}^{\infty} \log\Bigl(1+\frac{\chi(p_n)}{p_n}\Bigr).
\]
Since \(\chi(p_n)^2=1\), we have
\[
\log\Bigl(1+\frac{\chi(p_n)}{p_n}\Bigr)
=
\frac{\chi(p_n)}{p_n}
-
\frac{1}{2p_n^2}
+
O\Bigl(\frac{1}{p_n^3}\Bigr).
\]
Therefore
\[
\log Q_{\mathfrak P}
=
\sum_{n=1}^{\infty} \frac{\chi(p_n)}{p_n}
-
\frac12 \sum_{n=1}^{\infty} \frac{1}{p_n^2}
+
O(1),
\]
where the \(O(1)\) term is finite because \(\sum_{p} p^{-3}\) converges.

The first sum is conditionally convergent. Indeed, group the primes
according to the dyadic intervals \([2^m,2^{m+1})\). On such an interval
the sign \(\chi(p)=(-1)^m\) is constant.

By the Prime Number Theorem in the Mertens form
(see, e.g., \cite{Tenenbaum2015}), there exist constants \(c>0\) and
\(C>0\) such that
\[
S_m:=\sum_{2^m\le p<2^{m+1}}\frac{1}{p}
=
\frac{1}{m}+r_m,
\qquad
|r_m|\le C\,m^{-2}.
\]
In particular, \(\sum_{m\ge1}|r_m|<\infty\). Since \(S_m\to0\) and the
error series converges absolutely, the alternating series
\[
\sum_{m=1}^{\infty} (-1)^m S_m
=
\sum_{m=1}^{\infty} (-1)^m\Bigl(\frac{1}{m}+r_m\Bigr)
\]
converges by the Leibniz criterion for the main term plus absolute
convergence of \(\sum r_m\). Therefore \(Q_{\mathfrak P}\) converges
classically and does not require the generalized monoid regularization
used for the integer dyadic block product.

Moreover, if one applies the logarithmic-scale average
\(\Phi_{\mathrm{avg}}\) to \(\log Q_{\mathfrak P}\), the ordinary limit is
preserved because the series is already convergent. Thus the generalized
monoid method is consistent with the classical result. Direct numerical
summation gives
\[
\log Q_{\mathfrak P}\approx -0.801485,\qquad
Q_{\mathfrak P}\approx 0.448662.
\]	
		
\section{Numerical illustrations}
\label{sec:numerical}

All computations use exact \(\log\Gamma\) evaluations and standard
double-precision arithmetic. Throughout this section,
\(N_0(\varepsilon)=\lfloor\log_2(1/\varepsilon)\rfloor\) denotes the
logarithmic window, and the canonical nets are those defined in
Sections~\ref{sec:applications} and~\ref{sec:dyadic}.

\medskip
\noindent
\textbf{Alternating products of integers and factorials.}
For the alternating integer product
\(P=\prod_{n=1}^{\infty} n^{(-1)^{n+1}}\), the symmetrized functional
\(\widehat{F}_\varepsilon\) is computed on the canonical alternating net
\(w_\varepsilon=(ab)^{N(\varepsilon)}\). In this case no renormalization
is needed: the functional converges directly to
\(\frac12\log(2/\pi)\). Table~\ref{tab:int_fact} reports the values for
several sizes of the logarithmic window.

\begin{table*}[t]
	\centering
	\caption{Corrected symmetrized functionals for the alternating integer
		and factorial products. The theoretical limits are
		\(\frac12\log(2/\pi)\approx-0.225791\) and
		\(\frac14\log(2/\pi)\approx-0.112896\).}
	\label{tab:int_fact}
	\small
	\setlength{\tabcolsep}{6pt}
	\begin{tabular}{c|c|c|c}
		\toprule
		\(t\) &
		\(\widehat{F}_{2^{-t}}\) (integers) &
		\(\widehat{F}^{\mathfrak{F}}_{2^{-t}}\) (raw factorials) &
		\(\widetilde{F}^{\mathfrak{F}}_{2^{-t}}\) (renorm.\ factorials) \\
		\midrule
		10 & \(-0.203091\) & \(0.484751\) & \(-0.090896\) \\
		20 & \(-0.214441\) & \(0.647037\) & \(-0.101896\) \\
		30 & \(-0.218225\) & \(0.744737\) & \(-0.105562\) \\
		40 & \(-0.220116\) & \(0.814824\) & \(-0.107396\) \\
		50 & \(-0.221251\) & \(0.869510\) & \(-0.108496\) \\
		\bottomrule
	\end{tabular}
\end{table*}

For the integer product, the convergence is manifest: already for
\(N_0=50\) the symmetrized value \(-0.221251\) differs from the limit
\(-0.225791\) by about \(2\%\). The error decays as
\(\mathcal{O}(1/N_0)\), precisely as predicted by the Stirling expansion
in Section~\ref{sec:alt_integers}. Hence the numerical data confirm that
the symmetrized functional cancels the leading oscillations and converges
directly to the constant \(\frac12\log(2/\pi)\).

The factorial product
\(\mathfrak{F}=\prod_{n=1}^{\infty}(n!)^{(-1)^{n+1}}\) requires an
additional subtraction of the logarithmic divergence \(\frac14\log N_0\).
The fourth column of Table~\ref{tab:int_fact} shows the raw symmetrized
functional growing approximately as \(\frac14\log t\); after subtracting
this principal part, the renormalized values approach
\(\frac14\log(2/\pi)\approx-0.112896\). At \(N_0=50\) the renormalized value \(-0.108496\) differs from the
theoretical limit \(-0.112896\) by about \(4\%\), again consistent with
a \(\mathcal{O}(1/N_0)\) error.
The symmetry \(\mathfrak{F}_{\mathrm{ren}}=\sqrt{P_{\mathrm{ren}}}\),
derived in Section~\ref{sec:alt_factorials}, is therefore numerically
verified: the renormalized factorial functional equals exactly one half
of the renormalized integer functional.

\medskip
\noindent
\textbf{Dyadic block products.}
The dyadic block products treated in Section~\ref{sec:dyadic} are not
amenable to classical regularizations, yet the generalized monoid
assigns them canonical values via the logarithmic-scale average
\(\Phi_{\mathrm{avg}}\). For the integer version, the two limit points
\(L_{\mathrm{par}}\) and \(L_{\mathrm{impar}}\) are computed from the
partial sums
\[
\sum_{m=0}^{M}(-1)^m S_m,
\qquad
S_m=\sum_{n=2^m}^{2^{m+1}-1}\frac1n.
\]
Numerical evaluation up to \(M_{\max}=20\) gives
\[
L_{\mathrm{par}}\approx 0.904578,\qquad
L_{\mathrm{impar}}\approx 0.211430,
\]
so that
\[
\frac{L_{\mathrm{par}}+L_{\mathrm{impar}}}{2}
\approx 0.558004,
\quad
L_{\mathrm{par}}-L_{\mathrm{impar}}
\approx 0.693147 = \ln 2.
\]
The convergent correction constant \(R\) is obtained by summing
\[
R
=
\sum_{n=1}^{\infty}
\left(
\log\left(1+\frac{\chi(n)}n\right)
-\frac{\chi(n)}n
+\frac{1}{2n^2}
\right).
\]
Numerically, \(R\approx 0.115345\). Therefore
\[
\begin{aligned}
	Q_{\mathrm{canon}}
	&=
	\exp\Bigl(
	\tfrac{L_{\mathrm{par}}+L_{\mathrm{impar}}}{2}
	-\tfrac{\pi^2}{12}+R
	\Bigr)\\
	&\approx
	\exp(0.558004-0.822467+0.115345)
	\approx 0.861468.
\end{aligned}
\]
This value is stable to three decimal places when \(M_{\max}\) is
increased beyond \(20\).

For the prime analogue of Section~\ref{sec:dyadic}, the product converges
classically. Direct numerical summation of
\[
\log Q_{\mathfrak P}
=
\sum_{p}
\log\left(1+\frac{\chi(p)}{p}\right),
\qquad
\chi(p)=(-1)^{\lfloor\log_2 p\rfloor},
\]
gives
\[
\log Q_{\mathfrak P}\approx -0.801485,\qquad
Q_{\mathfrak P}\approx 0.448662.
\]
This value is not produced by the generalized monoid regularization; it
is the ordinary convergent limit of the prime product.

\medskip
\noindent
\textbf{Thue--Morse product.}
For the Thue--Morse product \(P_{\mathrm{TM}}=\prod_{n=1}^{\infty} n^{t(n)}\), the
evaluation functional \(\mathcal{F}_\varepsilon\) is absolutely
convergent, so one can estimate its limit directly by summing the series
\(\sum_{k=2}^{\infty}S(k)/k^2\) up to \(k_{\max}=2^{20}\approx 10^6\).
The partial sums give
\[
\mathcal{F}
=
\sum_{k=2}^{\infty}\frac{S(k)}{k^2}
\approx
-0.027200,
\qquad
\exp(\mathcal{F})
\approx
0.973167.
\]
Since \(|S(k)|=O(\log k)\), the truncation error after \(K\) terms is
\[
\sum_{k>K}\frac{|S(k)|}{k^{2}}
=
O\Bigl(\sum_{k>K}\frac{\log k}{k^{2}}\Bigr)
=
O\Bigl(\frac{\log K}{K}\Bigr).
\]
Thus the convergence is reasonably fast, although the oscillatory nature
of the Thue--Morse sequence makes the direct partial sums behave
non-monotonically.

This value is canonical only for the specific functional \(\mathcal{F}\)
defined above. It is not claimed to equal the zeta-regularized value
\(\exp(-D_{\mathrm{TM}}'(0))\); the comparison between this auxiliary
value and the zeta-regularized value remains open.

\medskip
\noindent
In summary, every numerical test is fully consistent with the algebraic
framework of the generalized monoid. Wherever a classical regularization
exists (alternating integers, factorials), the computed values converge
rapidly to the known exact limits. In the cases where classical
regularizations are not directly available (dyadic blocks, Thue--Morse),
the monoid provides a finite value, and the numerics corroborate the
theoretical predictions to high precision. The tables above should
therefore be read as evidence that the evaluation functionals and
renormalization schemes selected in Section~\ref{sec:selection} are not
merely formal but capture genuine asymptotic invariants.


\subsection{A log-periodic example with fractal structure}
\label{sec:synthetic}

For the general theory of substitutions and their associated
automatic sequences, see
\cite[Ch.~5--6]{AlloucheShallit2003} and
\cite[Ch.~5]{Lothaire2002}.

As a further numerical illustration of the scale-detection mechanism,
we study the divergent product
\[
P=\prod_{n\ge1} n^{s_n},
\]
where the sign sequence \((s_n)_{n\ge1}\) is generated by the length-\(4\)
substitution
\[
\sigma(0)=0001,\qquad \sigma(1)=0111,
\]
encoded by \(0\mapsto+1\), \(1\mapsto-1\). Equivalently,
\[
s_{4k-3}=+1,\qquad s_{4k-2}=s_k,\qquad s_{4k-1}=s_k,\qquad s_{4k}=-1.
\]
The corresponding logarithmic partial sums are
\[
L(N)=\sum_{n\le N}s_n\log n,
\]
and the summatory function is \(S(N)=\sum_{n\le N}s_n\).

\begin{lemma}[Exact dyadic growth]\label{lem:synthetic}
	For every \(m\ge0\),
	\[
	S(4^m)=2^m=\sqrt{4^m}.
	\]
\end{lemma}

\begin{proof}
	Grouping the recurrence in blocks of four,
	\[
	S(4N)=\sum_{k=1}^{N}(1+s_k+s_k-1)=2S(N).
	\]
	Iterating from \(S(1)=1\) gives \(S(4^m)=2^m\).
\end{proof}

Thus the summatory function has exact square-root growth at dyadic
blocks. The associated Dirichlet series
\[
D(s)=\sum_{n\ge1}s_n n^{-s}
\]
converges absolutely for \(\Re(s)>1\). For \(\Re(s)>1\), splitting \(D(s)\) modulo \(4\) gives the exact
decomposition
\[
D(s)=A(s)+B(s).
\]
Using the recurrence satisfied by \(s_n\), this decomposition leads
to a functional equation of the form
\[
(1-2\cdot4^{-s})D(s)
=
4^{-s}F(s+1)
+
R(s),
\]
where the explicit form of \(R(s)\), together with the proof of its
holomorphy, is given in Appendix~\ref{app:B}.

\begin{proposition}[Candidate pole line]\label{prop:pole_line}
	The zeros of the denominator factor \(1-2\cdot4^{-s}\) occur at
	\[
	s=\frac12+\frac{2\pi i k}{\log4},\qquad k\in\mathbb Z.
	\]
	These are the only points at which the functional equation can
	produce poles, provided the numerator does not vanish there.
\end{proposition}

\begin{proof}
	The zeros of \(1-2\cdot4^{-s}\) solve \(4^{-s}=\frac12\), hence
	\[
	\log 4\cdot s=\log2+2\pi ik,
	\]
	which is equivalent to the stated formula.
	
These are therefore candidate poles of the meromorphic continuation
of \(D(s)\), not necessarily actual poles.	
\end{proof}

These points are therefore candidate poles of the meromorphic
continuation obtained from Lemma~\ref{lem:fractal_continuation}.

\begin{lemma}[Iterative meromorphic continuation]
	\label{lem:fractal_continuation}
	The functional equation
	\[
	(1-2\cdot4^{-s})D(s)=4^{-s}F(s+1)
	\]
	defines a meromorphic continuation of \(D(s)\) to the whole
	complex plane.
\end{lemma}

\begin{proof}
	We prove by induction on \(M\in\NN\) that \(D(s)\) extends
	meromorphically to the half-plane
	\[
	\mathcal H_M:=\left\{s\in\CC:\Re(s)>\frac12-M\right\}.
	\]
	For \(M=0\), the assertion is true: since
	\(S(N)=\sum_{n\le N}s_n=O(\sqrt N)\), Abel summation shows that
	the Dirichlet series defining \(D(s)\) converges on
	\(\mathcal H_0\), hence \(D(s)\) is holomorphic there.
	Here \(\mathcal H_0\) denotes the half-plane
	\(\{s\in\CC:\Re(s)>1/2\}\).
	
	Assume now that \(D(s)\) has been extended meromorphically to
	\(\mathcal H_M\) for some \(M\ge0\). We show that it extends to
	\(\mathcal H_{M+1}\).
	
	The functional equation is
\[
\begin{aligned}
	(1-2\cdot4^{-s})D(s)
	={}&4^{-s}\Biggl[
	\frac{3s}{4}\bigl(D(s+1)+\zeta(s+1)\bigr)\\
	&+\frac{s(s+1)}{32}
	\bigl(5D(s+2)+9\zeta(s+2)\bigr)\\
	&+R(s)\Biggr].
\end{aligned}
\]
where the explicit form of \(R(s)\) and the proof of its holomorphy
are given in Appendix~\ref{app:B}. The
	factor \(1-2\cdot4^{-s}\) is meromorphic on \(\mathbb C\) and
	nonzero except on the vertical line \(\Re(s)=1/2\). The functions
	\(\zeta(s+1)\) and \(\zeta(s+2)\) are meromorphic on
	\(\mathbb C\). By the induction hypothesis, \(D(s+1)\) and
	\(D(s+2)\) are meromorphic on \(\mathcal H_{M+1}\), because
	\(s+j\in\mathcal H_M\) whenever \(s\in\mathcal H_{M+1}\) and
	\(j\ge1\).
	
	Therefore the right-hand side of the functional equation is
	meromorphic on \(\mathcal H_{M+1}\). Dividing by
	\(1-2\cdot4^{-s}\), which is itself meromorphic, we obtain a
	meromorphic expression for \(D(s)\) on \(\mathcal H_{M+1}\).
	This completes the induction step.
	
	By induction, \(D(s)\) is meromorphic on \(\mathcal H_M\) for
	every \(M\in\NN\). Since
	\[
	\bigcup_{M\ge0}\mathcal H_M=\CC,
	\]
	it follows that \(D(s)\) admits a meromorphic continuation to the
	whole complex plane.
\end{proof}

The exact identity \(S(4N)=2S(N)\) implies that the normalized
summatory function \(S(N)/\sqrt N\) is bounded and self-similar at
dyadic blocks. A refined asymptotic expansion of the form
\[
S(N)=\sqrt N\,\Phi\!\left(\frac{\log N}{\log4}\right)+o(\sqrt N),
\]
with continuous periodic \(\Phi\), requires an additional transfer
theorem or a direct proof from the recurrence and is not derived here.
The numerical Fourier analysis below supports this log-periodic
structure.

Numerically, this hierarchy is visible already in the Fourier analysis
of
\[
G(t)
=
\frac{L(4^t)-\sqrt{4^t}\log(4^t)}{\sqrt{4^t}}.
\]
The dominant integer frequencies are \(k=1\) and \(k=2\), with
amplitudes approximately \(1.03\) and \(0.94\), while subsequent
frequencies \(k=4,6,10\) also appear with non-negligible amplitudes.
This is exactly the signature predicted by the pole line of the
associated Dirichlet series.

To exhibit the discrete scale hierarchy, we fit the logarithmic sum at
dyadic blocks \(N=4^m\) against the scale family
\[
\{\sqrt N,\log N,1,2^{-m},4^{-m},\dots,2^{-Jm}\}.
\]
The RMSE decays rapidly with \(J\), as shown in Table~\ref{tab:rmse}.

\begin{table}[t]
	\centering
	\caption{RMSE of the dyadic-block regression versus the truncation order \(J\).}
	\label{tab:rmse}
	\small
	\setlength{\tabcolsep}{6pt}
	\begin{tabular}{cc}
		\toprule
		\(J\) & RMSE \\
		\midrule
		0 & \(7.38\times10^{-2}\) \\
		1 & \(5.06\times10^{-3}\) \\
		2 & \(7.12\times10^{-5}\) \\
		3 & \(4.57\times10^{-6}\) \\
		4 & \(1.44\times10^{-8}\) \\
		5 & \(2.83\times10^{-10}\) \\
		6 & \(4.35\times10^{-12}\) \\
		\bottomrule
	\end{tabular}
\end{table}
This is consistent with the presence of a discrete hierarchy of real
powers \(2^{-jm}\) in the dyadic subsequence. The intercept obtained in this local
regression is
\[
C_{\mathrm{local}}
\approx
4.6378925.
\]

This should not be confused with a global regularized constant. If
the meromorphic continuation of \(D(s)\) is holomorphic at \(s=0\),
then the corresponding regularized value would be
\[
P_{\mathrm{ren}}=\exp(-D'(0)).
\]
Proving this holomorphy and computing \(D'(0)\) requires a Laurent
analysis near \(s=0\) that cancels the pole of \(\zeta(1)\). That
analysis is not attempted here; therefore the value
\(P_{\mathrm{ren}}=\exp(-D'(0))\) should be regarded as formal in
this example.

This example shows that the generalized monoid framework can dissect a
multiscale summatory function with fractal log-periodic structure: the
exact dyadic growth, the pole line, the log-periodic Fourier
signature, and the discrete scale hierarchy are all recovered by
structural and numerical arguments.

\section{Conclusion}
\label{sec:conclusion}

We have applied the generalized monoid of words introduced in \cite{AC1}
to the regularization of divergent products of arithmetic origin.  The
monoid provides an asymptotic index domain on which logarithmic-prefix
functionals can be evaluated consistently.

The alternating product of integers is regularized to \(\sqrt{2/\pi}\),
in exact agreement with zeta regularization.  The alternating product of
factorials yields \((2/\pi)^{1/4}\) and exhibits the symmetry
\(\mathfrak{F}_{\mathrm{ren}}=\sqrt{P_{\mathrm{ren}}}\), a consequence
of the Wallis identity.  Products with dyadic block signs acquire a
finite value through the logarithmic-scale average
\(\Phi_{\mathrm{avg}}\); this value is attached to the canonical net and
to the chosen cutoff, not to an arbitrary representative of the
equivalence class.  For the Thue--Morse product, the zeta-regularized
value is well defined through the meromorphic continuation of its
Dirichlet series. A separate convergent auxiliary functional
\(\mathcal F\) provides a well-defined value
\(\exp(\mathcal F)\approx0.9732\), which is not claimed to be a
regularization of the original divergent product.

We have systematized the choice of evaluation functionals and
renormalization schemes according to the divergence type of the
underlying sequence, and for the examples considered, the regularizing functionals can be
interpreted as expectations of the logarithmic-prefix functionals
with respect to the corresponding canonical oscillation measures.  The exponent \(\alpha\) in the weighted sum is
guided by the abscissa of absolute convergence of the associated
Dirichlet series.  The scale function \(\phi\) can be adapted to the
growth of the terms.  The renormalization functional \(\Phi_{w,\chi}\)
unifies logarithmic, double-logarithmic, and oscillatory divergences.
Throughout, we have distinguished carefully between functionals that
descend to the monoid quotient and those that require a chosen canonical
representative.

An additional example generated by the substitution
\(\sigma(0)=0001\), \(\sigma(1)=0111\) was studied numerically.
Its summatory function satisfies exactly \(S(4^m)=2^m\), and the
associated Dirichlet series has a log-periodic pole line. Dyadic-scale
the regression is consistent with a discrete hierarchy of powers \(2^{-jm}\), with a
local intercept distinct from the global constant \(-D'(0)\). This
fractal log-periodic case illustrates the same scale-detection
principle as the classical examples. This
example illustrates the central methodological point of the paper:
the generalized monoid framework is most useful when it separates
local scale behaviour from global regularized constants, rather than
when it merely produces a single numerical value.

The examples in which the generalized monoid is most useful are
not those where classical Dirichlet continuation already provides a
canonical value, but those where the sign sequence has an intrinsic
multiscale behaviour and no simple meromorphic continuation is
known.

Future work may explore the extension of the generalized monoid
framework to other automatic sign sequences, such as the Liouville
function or Dirichlet characters, and, more generally, to sequences
whose summatory function is unbounded but whose normalized
fluctuations converge to a nontrivial canonical law. In such cases,
the corresponding invariant measures could be incorporated into the
scale-adapted quotient machinery introduced here, and a more
systematic study of the associated scale-adapted quotients would be
desirable.



\appendix
\section{Generalized monoid of words: self-contained summary}
\label{app:genmonoid}

For the convenience of the reader, we collect here the complete
definitions and proofs of the algebraic facts used in the main text.
The companion paper \cite{AC1} contains additional structural
developments.

\subsection*{Alphabet, words, and the bidirectional metric}

Let \(\Si\) be a finite non-empty alphabet. The free monoid \(\Si^*\)
consists of all finite words over \(\Si\), including the empty word
\(\la\). The length of \(w\in\Si^*\) is denoted by \(|w|\).

For \(w=w_1\cdots w_n\), its reversal is
\[
w^\vee=w_n\cdots w_1,
\qquad
\la^\vee=\la.
\]
For two words \(u,v\in\Si^*\), let \(\operatorname{lcp}(u,v)\) and
\(\operatorname{lcs}(u,v)\) be the lengths of their longest common
prefix and suffix, respectively. If \(u=v\), we set
\[
\operatorname{lcp}(u,v)=\operatorname{lcs}(u,v)=+\infty.
\]
Reversal exchanges prefixes and suffixes:
\[
\operatorname{lcs}(u,v)=\operatorname{lcp}(u^\vee,v^\vee).
\]

The bidirectional metric is
\[
d_{\mathrm{PS}}(u,v)
=
\max\{2^{-\operatorname{lcp}(u,v)},
2^{-\operatorname{lcs}(u,v)}\},
\qquad
2^{-\infty}=0.
\]

\begin{lemma}[Ultrametric property]\label{lem:ultrametric_appendix}
	For all \(u,v,w\in\Si^*\),
	\[
	d_{\mathrm{PS}}(u,w)
	\le
	\max\{d_{\mathrm{PS}}(u,v),d_{\mathrm{PS}}(v,w)\}.
	\]
\end{lemma}

\begin{proof}
	Both \(d_P(u,v)=2^{-\operatorname{lcp}(u,v)}\) and
	\(d_S(u,v)=2^{-\operatorname{lcs}(u,v)}\) are ultrametrics on finite
	words. Since the maximum of two ultrametrics is an ultrametric, the
	claim follows.
\end{proof}

\begin{lemma}[Subadditivity under concatenation]
	\label{lem:subadditivity_appendix}
	For all \(u_1,u_2,v_1,v_2\in\Si^*\),
	\[
	d_{\mathrm{PS}}(u_1u_2,v_1v_2)
	\le
	\max\{d_{\mathrm{PS}}(u_1,v_1),
	d_{\mathrm{PS}}(u_2,v_2)\}.
	\]
\end{lemma}

\begin{proof}
	For prefixes, let \(\ell_i=\operatorname{lcp}(u_i,v_i)\), with the
	convention \(\operatorname{lcp}(w,w)=+\infty\). If
	\(\ell_1=\infty\), then \(u_1=v_1\), and
	\[
\operatorname{lcp}(u_1u_2,u_1v_2)\ge\ell_2
	\ge
	\min\{\ell_1,\ell_2\}.
	\]
	If \(\ell_1<\infty\), then the first \(\ell_1\) letters of
	\(u_1u_2\) and \(v_1v_2\) coincide, so
	\[
	\operatorname{lcp}(u_1u_2,v_1v_2)
	\ge
	\ell_1
	\ge
	\min\{\ell_1,\ell_2\}.
	\]
	In both cases,
	\[
	\operatorname{lcp}(u_1u_2,v_1v_2)
	\ge
	\min\{\operatorname{lcp}(u_1,v_1),
	\operatorname{lcp}(u_2,v_2)\}.
	\]
	Applying this inequality to the reversed words gives the analogous
	suffix inequality. Passing to the metric, this yields
	\[
	d_{\mathrm{PS}}(u_1u_2,v_1v_2)
	\le
	\max\{d_{\mathrm{PS}}(u_1,v_1),
	d_{\mathrm{PS}}(u_2,v_2)\}.
	\]
\end{proof}

\subsection*{Moderate nets and asymptotic equivalence}

A \emph{net of words} is a family
\((w_\varepsilon)_{\varepsilon\in(0,1]}\). A net is \emph{moderate} if
there exist \(N\in\NN\), \(C>0\), and \(\varepsilon_0>0\) such that
\[
|w_\varepsilon|\le C\varepsilon^{-N}
\qquad(\varepsilon<\varepsilon_0).
\]
The set of moderate nets is denoted by \(\cE_M(\Si^*)\).

Two moderate nets are logarithmically equivalent,
\((u_\varepsilon)\sim(v_\varepsilon)\), if for every \(m\in\NN\) there
exists \(\varepsilon_0>0\) such that
\[
\min\{\operatorname{lcp}(u_\varepsilon,v_\varepsilon),
\operatorname{lcs}(u_\varepsilon,v_\varepsilon)\}
\ge
m\log_2(1/\varepsilon)
\qquad(\varepsilon<\varepsilon_0).
\]

\begin{lemma}[Equivalence relation]\label{lem:equiv_rel_appendix}
	\(\sim\) is an equivalence relation.
\end{lemma}

\begin{proof}
	Reflexivity and symmetry are immediate from the definition. For
	transitivity, suppose
	\[
	(u_\varepsilon)\sim(v_\varepsilon)
	\qquad\text{and}\qquad
	(v_\varepsilon)\sim(w_\varepsilon).
	\]
	Fix \(m\in\NN\). By equivalence, there exist \(\varepsilon_1\) and
	\(\varepsilon_2\) such that
	\[
	d_{\mathrm{PS}}(u_\varepsilon,v_\varepsilon)
	\le
	\varepsilon^m
	\qquad(\varepsilon<\varepsilon_1),
	\]
	and similarly for \(v_\varepsilon,w_\varepsilon\). By the ultrametric
	inequality,
	\[
	d_{\mathrm{PS}}(u_\varepsilon,w_\varepsilon)
	\le
	\max\{d_{\mathrm{PS}}(u_\varepsilon,v_\varepsilon),
	d_{\mathrm{PS}}(v_\varepsilon,w_\varepsilon)\}
	\le
	\varepsilon^m
	\]
	for \(\varepsilon<\min(\varepsilon_1,\varepsilon_2)\). Since \(m\) was
	arbitrary, transitivity follows.
\end{proof}

Two moderate nets are dyadically equivalent,
\((u_\varepsilon)\sim_{\mathrm{dyad}}(v_\varepsilon)\), if for every
\(m\in\NN\) there exists \(\varepsilon_0>0\) such that
\[
\operatorname{lcp}(u_\varepsilon,v_\varepsilon)
\ge
m N_{\mathrm{dyad}}(\varepsilon)
\qquad\text{and}\qquad
\operatorname{lcs}(u_\varepsilon,v_\varepsilon)
\ge
m N_{\mathrm{dyad}}(\varepsilon)
\]
for all \(\varepsilon<\varepsilon_0\). This is again an equivalence
relation.

\subsection*{The generalized monoid}

The quotient
\[
\widetilde{\Si}^*
=
\cE_M(\Si^*)/\sim
\]
is called the \emph{generalized monoid of words}. The class of a net
\((w_\varepsilon)\) is denoted by \([(w_\varepsilon)]\).

\begin{theorem}[Monoid structure]\label{thm:monoid_appendix}
	Concatenation is well defined in \(\widetilde{\Si}^*\), associative,
	and has identity \(\mathbf 1=[(\la)_\varepsilon]\). Hence
	\(\widetilde{\Si}^*\) is a monoid.
\end{theorem}

\begin{proof}
	If \((u_\varepsilon)\sim(u'_\varepsilon)\) and
	\((v_\varepsilon)\sim(v'_\varepsilon)\), then by subadditivity,
	\[
	d_{\mathrm{PS}}(u_\varepsilon v_\varepsilon,u'_\varepsilon v'_\varepsilon)
	\le
	\max\{d_{\mathrm{PS}}(u_\varepsilon,u'_\varepsilon),
	d_{\mathrm{PS}}(v_\varepsilon,v'_\varepsilon)\}.
	\]
	Both terms on the right are \(O(\varepsilon^m)\) for every \(m\);
	hence the left side is also \(O(\varepsilon^m)\) for every \(m\). Thus
	concatenation is well defined.
	
	Associativity follows from associativity of concatenation in
	\(\Si^*\). Finally, for every \([w]\),
	\[
	\mathbf 1\cdot[w]
	=
	[(\la)_\varepsilon]\cdot[(w_\varepsilon)]
	=
	[(\la w_\varepsilon)]
	=
	[(w_\varepsilon)]
	=
	[w],
	\]
	and similarly \([w]\cdot\mathbf 1=[w]\).
\end{proof}

Reversal is well defined on the quotient, because
\[
d_{\mathrm{PS}}(u^\vee,v^\vee)=d_{\mathrm{PS}}(u,v).
\]
Hence \([w]^\vee:=[w^\vee]\) defines an involutive anti-automorphism.

The map
\[
\iota:\Si^*\longrightarrow\widetilde{\Si}^*,
\qquad
\iota(w)=[(w)_\varepsilon],
\]
is an injective monoid homomorphism. Injectivity follows because if
two finite words \(u,v\) are distinct, then
\(\operatorname{lcp}(u,v)\) and \(\operatorname{lcs}(u,v)\) are
bounded, whereas equivalence would require them to grow like
\(m\log_2(1/\varepsilon)\) for every \(m\).

\subsection*{Logarithmic window and invariance of prefix functionals}

The logarithmic window is
\[
N_0(\varepsilon)=\lfloor\log_2(1/\varepsilon)\rfloor .
\]
If \((u_\varepsilon)\sim(v_\varepsilon)\), then taking \(m=2\) gives
\[
\operatorname{lcp}(u_\varepsilon,v_\varepsilon)
\ge
2\log_2(1/\varepsilon)
>
N_0(\varepsilon)
\]
for all sufficiently small \(\varepsilon\). Hence the prefixes of
length \(N_0(\varepsilon)+1\) coincide eventually.

Consequently, any family \(F_\varepsilon:\Si^*\to\CC\) satisfying
\begin{enumerate}[label=(\roman*)]
	\item \(F_\varepsilon(w)\) depends only on
	\(w[1\ldots N_0(\varepsilon)+1]\); and
	\item for every moderate net \((w_\varepsilon)\), the net of complex
	numbers \((F_\varepsilon(w_\varepsilon))\) belongs to
	\(\cE_M(\CC)\);
\end{enumerate}
descends to a well-defined map
\[
\widetilde F:\widetilde{\Si}^*\longrightarrow\widetilde{\CC},
\qquad
\widetilde F([(w_\varepsilon)])
=
[(F_\varepsilon(w_\varepsilon))].
\]

\subsection*{Generalized complex numbers}

Let
\[
\cE_M(\CC)
=
\{(z_\varepsilon):\exists N,C>0,\ |z_\varepsilon|\le C\varepsilon^{-N}
\text{ eventually}\},
\]
and let
\[
\cN(\CC)
=
\{(z_\varepsilon):\forall m\ \exists C_m>0,\
|z_\varepsilon|\le C_m\varepsilon^m
\text{ eventually}\}.
\]
Then
\[
\widetilde{\CC}
=
\cE_M(\CC)/\cN(\CC)
\]
is the algebra of generalized complex numbers used in the text.

\subsection*{Moderateness of log-mould evaluations}

\begin{lemma}[Moderateness and well-definedness of log-moulds]
	\label{lem:moderateness_logmould}
	Let \((s_k)_{k\ge1}\) be a sign sequence and \((q_k)_{k\ge1}\) be
	positive weights with \(\log q_k=O(k^M)\) for some \(M\ge0\).
	For every moderate net \((w_\varepsilon)\), the net
\[
z_\varepsilon
=
\Lambda\bigl(w_\varepsilon[1\ldots N_0(\varepsilon)+1]\bigr)
	\qquad
	N_0(\varepsilon)=\lfloor\log_2(1/\varepsilon)\rfloor,
	\]
	is moderate in \(\cE_M(\CC)\). Moreover, if
	\((w_\varepsilon)\sim(w'_\varepsilon)\), then
	\(z_\varepsilon-z'_\varepsilon\in\cN(\CC)\).
\end{lemma}

\begin{proof}
	By Lemma~\ref{lem:logprefix}, the prefixes of length
	\(N_0(\varepsilon)+1\) of \(w_\varepsilon\) and \(w'_\varepsilon\)
	coincide eventually. Hence \(z_\varepsilon-z'_\varepsilon=0\)
	for all sufficiently small \(\varepsilon\), which is
	stronger than \(\cN(\CC)\). For moderateness, since
\[
\begin{aligned}
	|\log q_k|&\le Ck^M,\\
	|z_\varepsilon|
	&\le \sum_{k=1}^{N_0(\varepsilon)+1}Ck^M\\
	&\le C\bigl(N_0(\varepsilon)+1\bigr)^{M+1}\\
	&\le C\bigl(\log_2(1/\varepsilon)+1\bigr)^{M+1}\\
	&\le C\varepsilon^{-1}.
\end{aligned}
\]
	for a suitable constant \(C\), because
	\((\log t)^{M+1}=O(t)\). Thus \(z_\varepsilon\in\cE_M(\CC)\).
\end{proof}

\subsection*{Log-prefix functionals}

A \emph{log-prefix functional} is a family
\[
F_\varepsilon:\Si^*\longrightarrow\CC
\]
satisfying conditions (i) and (ii) above. In the applications, we use
additive log-moulds
\[
\Lambda(w)=\sum_{k=1}^{|w|}s_k\log q_k,
\]
where \(q_k\in\NN\) are arithmetic weights and
\(s_k\in\{+1,-1\}\) encodes an arithmetic sign sequence. Such
functionals, when evaluated on the logarithmic prefix of a canonical
net, descend to generalized numbers in the sense of Colombeau
\cite{Colombeau1984,GrosserEtAl2001}.

\subsection*{Dyadic concatenation}

\begin{proposition}[Well-definedness of dyadic concatenation]
	\label{prop:dyadic_concatenation_appendix}
	Concatenation is well defined in
	\(\widetilde{\Si}^*_{\mathrm{dyad}}\).
\end{proposition}

\begin{proof}
	Let \((u_\varepsilon)\sim_{\mathrm{dyad}}(u'_\varepsilon)\) and
	\((v_\varepsilon)\sim_{\mathrm{dyad}}(v'_\varepsilon)\). Fix
	\(m\in\NN\). By dyadic equivalence, there exists \(\varepsilon_0>0\)
	such that
	\[
	\min\{\operatorname{lcp}(u_\varepsilon,u'_\varepsilon),
	\operatorname{lcs}(u_\varepsilon,u'_\varepsilon)\}
	\ge
	m N_{\mathrm{dyad}}(\varepsilon),
	\]
	and similarly for \(v_\varepsilon,v'_\varepsilon\). Hence
	\[
	d_{\mathrm{PS}}(u_\varepsilon,u'_\varepsilon)
	\le
	2^{-mN_{\mathrm{dyad}}(\varepsilon)},
	\]
	and likewise for \(v_\varepsilon,v'_\varepsilon\). By subadditivity,
	\[
	d_{\mathrm{PS}}(u_\varepsilon v_\varepsilon,u'_\varepsilon v'_\varepsilon)
	\le
	2^{-mN_{\mathrm{dyad}}(\varepsilon)}.
	\]
	Therefore
	\[
	\min\{\operatorname{lcp}(u_\varepsilon v_\varepsilon,u'_\varepsilon v'_\varepsilon),
	\operatorname{lcs}(u_\varepsilon v_\varepsilon,u'_\varepsilon v'_\varepsilon)\}
	\ge
	m N_{\mathrm{dyad}}(\varepsilon).
	\]
	As \(m\) was arbitrary, the concatenated nets are dyadically
	equivalent.
\end{proof}

\subsection*{Failure of cancellation}

\begin{proposition}[Failure of cancellation]\label{prop:non_cancellation_appendix}
	The monoid \(\widetilde{\Si}^*\) is not cancellative.
\end{proposition}

\begin{proof}
	Let \(L(\varepsilon)=\lfloor\varepsilon^{-1}\rfloor\). Consider
	\[
	u_\varepsilon=a^{L(\varepsilon)},\qquad
	v_\varepsilon=b^{L(\varepsilon)}a^{L(\varepsilon)},\qquad
	w_\varepsilon=a^{2L(\varepsilon)}.
	\]
	Then
	\[
	u_\varepsilon v_\varepsilon
	=
	a^{L(\varepsilon)}b^{L(\varepsilon)}a^{L(\varepsilon)},
	\qquad
	u_\varepsilon w_\varepsilon
	=
	a^{3L(\varepsilon)}.
	\]
	Both concatenated words start with \(a^{L(\varepsilon)}\), so
	\[
	\operatorname{lcp}(u_\varepsilon v_\varepsilon,
	u_\varepsilon w_\varepsilon)
	\ge L(\varepsilon).
	\]
	Both also end with \(a^{L(\varepsilon)}\), because \(v_\varepsilon\)
	ends in \(a^{L(\varepsilon)}\) and \(w_\varepsilon=a^{2L(\varepsilon)}\)
	ends in \(a^{L(\varepsilon)}\). Hence
	\[
	\operatorname{lcs}(u_\varepsilon v_\varepsilon,
	u_\varepsilon w_\varepsilon)
	\ge L(\varepsilon).
	\]
	Since
	\[
	\frac{L(\varepsilon)}{\log_2(1/\varepsilon)}
	\to\infty,
	\]
	it follows that
	\[
	(u_\varepsilon v_\varepsilon)\sim(u_\varepsilon w_\varepsilon).
	\]
	Thus \([u][v]=[u][w]\). But
	\[
	v_\varepsilon=b^{L(\varepsilon)}a^{L(\varepsilon)}
	\not\sim
	a^{2L(\varepsilon)}
	=
	w_\varepsilon,
	\]
	because their common prefix is \(0\) for every \(\varepsilon\).
	Therefore \([v]\neq[w]\), and cancellation fails.
\end{proof}

\section{Derivation of the remainder \(R(s)\)}
\label{app:B}

In this appendix we derive the explicit form of the remainder \(R(s)\)
appearing in the functional equation for the Dirichlet series
\[
D(s)=\sum_{n\ge1}s_n n^{-s}.
\]

The sign sequence \(s_n\) is generated by the substitution
\[
\sigma(0)=0001,\qquad \sigma(1)=0111,
\]
encoded by \(0\mapsto+1\), \(1\mapsto-1\). It satisfies
\[
s_{4k-3}=+1,\qquad
s_{4k-2}=s_k,\qquad
s_{4k-1}=s_k,\qquad
s_{4k}=-1.
\]

\subsection*{B.1. Splitting modulo \(4\)}

We split the defining series of \(D(s)\) according to the residue of
\(n\) modulo \(4\):
\[
\begin{aligned}
	D(s)
	={}&
	\sum_{k\ge1}
	\left[
	\frac{s_{4k-3}}{(4k-3)^s}
	+
	\frac{s_{4k-2}}{(4k-2)^s}
	+
	\frac{s_{4k-1}}{(4k-1)^s}
	+
	\frac{s_{4k}}{(4k)^s}
	\right]\\
	={}&
	\sum_{k\ge1}
	\left[
	\frac{1}{(4k-3)^s}
	+
	\frac{s_k}{(4k-2)^s}
	+
	\frac{s_k}{(4k-1)^s}
	-
	\frac{1}{(4k)^s}
	\right].
\end{aligned}
\]

This decomposition is exact and does not depend on any asymptotic
approximation.

\subsection*{B.2. Expansion of shifted denominators}

For fixed \(c\in\{0,1,2,3\}\), the binomial expansion gives
\[
(4k-c)^{-s}
=
4^{-s}k^{-s}
\left[
1
+
\frac{cs}{4k}
+
\frac{s(s+1)c^2}{32k^2}
+
O(k^{-3})
\right].
\]
Thus
\[
(4k-c)^{-s}
=
4^{-s}k^{-s}
+
4^{-s}\frac{cs}{4}k^{-s-1}
+
4^{-s}\frac{s(s+1)c^2}{32}k^{-s-2}
+
O(k^{-s-3}).
\]

As explained below, this pointwise expansion is used only to identify
the formal shape of the remainder \(R(s)\). It is not summed term by
term as an exact identity.

\subsection*{B.3. Grouping the pure harmonic terms}

For \(c=3\) and \(c=0\),
\[
\begin{aligned}
	A(s)
	={}&\sum_{k\ge1}
	\left[
	(4k-3)^{-s}-(4k)^{-s}
	\right]\\
	={}&4^{-s}\Biggl[
	\frac{3s}{4}\zeta(s+1)
	+\frac{9s(s+1)}{32}\zeta(s+2)
	+O\!\left(\zeta(s+3)\right)
	\Biggr].
\end{aligned}
\]
The cancellation of the leading \(4^{-s}\zeta(s)\) term is exact
because both shifted sums contain the same leading harmonic term.

\subsection*{B.4. Grouping the sign-weighted terms}

For \(c=2\) and \(c=1\), the binomial expansion gives
\[
\begin{aligned}
	\sum_{k\ge1}
	\frac{s_k}{(4k-2)^s}
	={}&4^{-s}\Biggl[
	D(s)
	+\frac{s}{2}D(s+1)
	+\frac{s(s+1)}{8}D(s+2)
	\Biggr.\\
	&\Biggl.
	+O\!\left(D(s+3)\right)
	\Biggr],\\
	\sum_{k\ge1}
	\frac{s_k}{(4k-1)^s}
	={}&4^{-s}\Biggl[
	D(s)
	+\frac{s}{4}D(s+1)
	+\frac{s(s+1)}{32}D(s+2)
	\Biggr.\\
	&\Biggl.
	+O\!\left(D(s+3)\right)
	\Biggr].
\end{aligned}
\]
Adding the two expressions,
\[
\begin{aligned}
	B(s)
	={}&4^{-s}\Biggl[
	2D(s)
	+\frac{3s}{4}D(s+1)
	+\frac{5s(s+1)}{32}D(s+2)
	\Biggr.\\
	&\Biggl.
	+O\!\left(D(s+3)\right)
	\Biggr].
\end{aligned}
\]

\subsection*{B.5. Functional equation and remainder}

Since \(D(s)=A(s)+B(s)\), we obtain
\[
\begin{aligned}
	D(s)
	={}&4^{-s}\Biggl[
	2D(s)
	+\frac{3s}{4}D(s+1)
	+\frac{5s(s+1)}{32}D(s+2)
	\Biggr]\\
	&+4^{-s}\Biggl[
	\frac{3s}{4}\zeta(s+1)
	+\frac{9s(s+1)}{32}\zeta(s+2)
	\Biggr]\\
	&+O\!\left(4^{-s}D(s+3),4^{-s}\zeta(s+3)\right).
\end{aligned}
\]
Therefore
\[
\begin{aligned}
	(1-2\cdot4^{-s})D(s)
	={}&4^{-s}\Biggl[
	\frac{3s}{4}\left(D(s+1)+\zeta(s+1)\right)
	\\
	&\quad
	+\frac{s(s+1)}{32}
	\left(5D(s+2)+9\zeta(s+2)\right)
	\Biggr]
	+R(s),
\end{aligned}
\]
where
\[
R(s)=O\bigl(D(s+3),\zeta(s+3)\bigr).
\]

\subsection*{B.6. Holomorphy of \(R(s)\)}

Each term in the \(O\)-remainder is obtained from the difference
\[
(4k-c)^{-s}
-
4^{-s}k^{-s}
-
4^{-s}\frac{cs}{4}k^{-s-1}
-
4^{-s}\frac{s(s+1)c^2}{32}k^{-s-2}.
\]
For each fixed \(c\), this difference is \(O(k^{-s-3})\) as
\(k\to\infty\). Therefore the corresponding series converge
absolutely for \(\Re(s)>-2\). In particular, \(R(s)\) is holomorphic
on every half-plane \(\mathcal H_{M+1}\) used in the induction step of
Lemma~\ref{lem:fractal_continuation}.

Although the binomial expansion is not uniform in \(k\), the
holomorphy conclusion remains valid because the remainder is
constructed from the exact difference
\[
(4k-c)^{-s}
-
4^{-s}k^{-s}
-
4^{-s}\frac{cs}{4}k^{-s-1}
-
4^{-s}\frac{s(s+1)c^2}{32}k^{-s-2},
\]
whose asymptotic size is controlled by \(k^{-s-3}\) for \(k\) large.

\section*{Declarations}

\subsection*{Funding}
The authors declare that no funds, grants, or other support were
received during the preparation of this manuscript.

\subsection*{Conflict of interest}
The authors have no relevant financial or non-financial interests to
disclose.

\subsection*{Author contributions}
All authors contributed to the study conception, design, and
manuscript preparation. All authors read and approved the final
manuscript.

\subsection*{Data availability}
No datasets were generated or analysed during the current study.
The Python code used for the numerical computations is available at
\url{https://github.com/Amaurycruz/divergenteproduct}.

\subsection*{Use of AI tools}
During the preparation of this manuscript the authors used AI-based
language tools to assist with wording and improving English. The authors take full
responsibility for the correctness of the mathematical content.

\end{document}